\documentclass[11pt]{article}
\usepackage{amssymb,amsmath,amsfonts,amsthm,mathtools,enumerate,color}
\usepackage[toc,page,title,titletoc,header]{appendix}
\usepackage{graphicx,subfig}
\graphicspath{ {./figures/} }

\usepackage{indentfirst}
\usepackage{multicol}
\usepackage{booktabs}
\usepackage{url}
\usepackage{epstopdf} 

\numberwithin{equation}{section}

\newtheorem{Theorem}{Theorem}[section]
\newtheorem{Lemma}[Theorem]{Lemma}
\newtheorem{Proposition}[Theorem]{Proposition}
\newtheorem{Definition}[Theorem]{Definition}

 \def\p{\partial} \def\nb{\nonumber}
\def \Vh0{\stackrel{\circ}{V}_h} \def\to{\rightarrow}
   
\def\Om{\Omega}   

\newcommand{\q}{\quad}   \newcommand{\qq}{\qquad}

\def\R{{\mathbb R}}

\def\Z{{\mathbb{Z}}}

\def\l{\label}  \def\f{\frac}  \def\fa{\forall}
\def\b{\beta}  \def\a{\alpha} 
 \def\t{\times}  
\def \Del{\Delta}
\def\ms{\medskip}  

\def\p{\partial}

\newcommand{\lc}
{\mathrel{\raise2pt\hbox{${\mathop<\limits_{\raise1pt\hbox
{\mbox{$\sim$}}}}$}}}

\newcommand{\gc}
{\mathrel{\raise2pt\hbox{${\mathop>\limits_{\raise1pt\hbox{\mbox{$\sim$}}}}$}}}

\newcommand{\ec}
{\mathrel{\raise2pt\hbox{${\mathop=\limits_{\raise1pt\hbox{\mbox{$\sim$}}}}$}}}

\def\bb{\begin{equation}} \def\ee{\end{equation}}

\def\beqn{\begin{eqnarray}}  \def\eqn{\end{eqnarray}}

\def\beqnx{\begin{eqnarray*}} \def\eqnx{\end{eqnarray*}}

\def\bn{\begin{enumerate}} \def\en{\end{enumerate}}

\def\bd{\begin{description}} \def\ed{\end{description}}

\makeatletter
\newenvironment{tablehere}
  {\def\@captype{table}}
  {}
\newenvironment{figurehere}
  {\def\@captype{figure}}
  {}
\makeatother
\newenvironment{aligncases}
    {\left\{ \begin{aligned} }
    {\end{aligned} \right.    }

\newcommand{\diff}[1]{\partial_{#1}}
\newcommand{\grad}{\nabla}

\newcommand{\gene}{\mathcal{L}}
\newcommand{\ex}{\mathbb{E}}
\newcommand{\inprod}[2]{\langle #1,#2\rangle}
\newcommand{\size}[1]{\Delta #1}

\newcommand{\dom}{D^\epsilon}

\newcommand{\norm}[2]{\|#1\|_{#2}}

\newcommand{\alg}[1]{\mathcal{F}_{#1}}

\DeclareMathOperator{\range}{Range}

\begin{document}

\title{Fully Discrete Schemes and Their Analyses \\ for Forward-Backward Stochastic Differential Equations
}
\author{
Kazufumi Ito\thanks{‡Department of Mathematics and Center for Research in Scientific Computation, North Carolina
State University, Raleigh, NC 27695.
({\tt kito@unity.ncsu.edu})
}
\quad
Yufei Zhang\thanks{Mathematical Institute, University of Oxford, Oxford OX2 6GG, UK.
({\tt Yufei.Zhang@maths.ox.ac.uk})
}
\quad
Jun Zou\thanks{Department of Mathematics, The
Chinese University of Hong Kong, Shatin, Hong Kong, China.
The work of this author 
was substantially supported by Hong Kong RGC General Research Fund (projects 14322516 and 14306814).
({\tt zou@math.cuhk.edu.hk})}
}
\date{}

\maketitle

\noindent\textbf{Abstract.} We propose some numerical schemes for forward-backward stochastic differential equations (FBSDEs)
based on a new fundamental concept of
transposition solutions. These schemes exploit time-splitting methods for the variation of constants formula of the associated partial differential equations and a discrete representation of the transition semigroups.
The convergence of the schemes is established for FBSDEs with uniformly Lipschitz drivers, locally Lipschitz and maximal monotone drivers.
Numerical experiments
are presented for several nonlinear financial derivative pricing problems to demonstrate the adaptivity and effectiveness of
the new schemes. The ideas here can be applied
to construct high-order schemes for FBSDEs with general Markov forward processes.


\ms

\medskip

\section{Introduction}

Since the seminal work by Pardoux-Peng \cite{pardoux1990adapted} about the unique solvability results for the nonlinear backward stochastic differential equations (BSDEs) in early 1990's, BSDEs and forward-backward stochastic differential equations (FBSDEs) have
become a significant tool in many areas of science, see  \cite{el1997backward,el1997bsdefin} and references therein, for more about the application of such equations in stochastic optimal control and option pricing in mathematical finance. As it is usually difficult to obtain analytic solutions of BSDEs, even for linear cases,
it is necessary to design effective and robust numerical methods for BSDEs, that is, to construct
the state process and the martingale integrand process numerically.

There are several technical issues in constructing efficient and accurate numerical methods for FBSDEs. In the Markovian case with path-independent terminal conditions, two main technical obstacles are the dimensionality and high-order approximation.
The former arises naturally as a consequence of the close relationship between FBSDEs and their corresponding partial differential equations (PDEs), where the "curse of dimensionality" is an inevitable difficulty for any numerical solver. On the other hand, as noticed in \cite{ma2008numerical}, although higher order approximations are available for initial value problems for SDEs (see e.g., \cite{kloeden2011numerical}), it is non-trivial to extend those methods to coupled FBSDEs, even for the case with smooth coefficients. In the non-Markovian case where the terminal conditions are allowed to depend on the entire history of the forward process, the main difficulty usually comes from the approximation of the martingale integrand, whose path regularity is necessary for the construction of a numerical scheme for BSDEs. 
All these intrinsic difficulties make the construction of numerical solutions to BSDEs a much more challenging problem compared to the ordinary initial value problems for SDEs.

Enormous efforts have been made in recent years to circumvent those fundamental difficulties in the numerical resolution of BSDEs.
The most existing numerical methods for BSDEs can be classified into two groups. 
By projecting the solution on the available information at each step, one group of algorithms goes backwards in time and solves the BSDEs and FBSDEs directly. This idea was employed in  \cite{bally1997approximation} and \cite{chevance1997numerical} respectively  to derive numerical schemes with a random time discretization under strong regularity condition ($C^4$) on the coefficients. More recently, a new notion of $L^2$-regularity on the control part of the solution was introduced in \cite{zhang2004numerical}, based on which numerical schemes with deterministic time partitions can be constructed \cite{bender2008time, bouchard2004discrete,gobet2010time} and a strong convergence in time can be established  under some weak regularity assumptions on
the coefficients. We note that "projecting the solution on the current information" means that the evaluation of conditional expectations is required at each step, and consequently, this type of schemes results in a high order nesting of conditional expectations as it works backwards through time. Therefore, an efficient estimation of the conditional expectations must be employed to
derive a fully discretized scheme and to avoid explosive computational costs  (see \cite{ bally2003quantization, bender2007forward, bouchard2004discrete, carriere1996valuation,delarue2006forward,gobet2005regression} for some  choices of these simulation-based estimators).

Alternatively, based on the theoretical connection between FBSDEs and  PDEs via the Feynman-Kac-type representation formulas, a four step scheme for FBSDEs was suggested in \cite{ma1994solving},
 and more recently developed in \cite{douglas1996numerical, milstein2006numerical, zhang2002discretizing}. The main idea of these
schemes is to express the solution of a FBSDE as a value function of the time and state
of the forward process, which can be obtained from 
a deterministic backward nonlinear parabolic PDE.

We remark that all the numerical schemes mentioned above are basically for BSDEs with
natural filtration, for which the existence of solutions is established mostly by the Martingale Representation
Theorem. So these methods may not be applicable if the underlying filtration of
BSDEs is more general than the natural filtration, for instance, if we aim to price an option
as an inside trader who acquires more knowledge than the public market. Motivated by these
applications, a new notion of transposition solutions was proposed in \cite{lu2013well} for
BSDEs, and it coincides with the strong solution when the filtration is natural. The well-posedness of
 transposition solutions to BSDEs was investigated without the Martingale Representation
Theorem, hence principally more flexible for general filtration.


In this work, we shall propose three easily implementable numerical algorithms for decoupled Markovian FBSDEs under the framework of transposition solutions. This fundamental framework enables us to construct various different schemes by taking different
test functions, and these algorithms are more
flexible than most existing schemes when
they applied to FBSDEs with general filtration. We demonstrate in this work that by taking suitable test functions, this framework leads to numerical schemes solving FBSDEs through a PDE approach. We remark that, instead of the frequently used direct discretization
of the corresponding PDEs \cite{ma2002numberical,ma2008numerical}, our schemes exploit the time-splitting approach
for the variation of constants formula of the associated PDEs, and represent the transition semigroups of It\^{o} diffusions by their transition probabilities among partition sets, which can be efficiently evaluated by Monte Carlo method. Similar representations can be carried out for the semigroup associated to a general Markov process, which leads to straightforward extensions of our algorithms to BSDEs driven by L\'{e}vy process. Furthermore, high-order schemes can be developed systematically based on a high-order splitting of the variation of constants formula.
Another major contribution of this work is the convergence analysis of the probabilistic numerical schemes for the associated semilinear PDEs. Through the help of the semigroup theory, we establish the convergence respectively for FBSDEs with both
Lipschitz and maximal monotone drivers.  
We also remark that, to the best of our knowledge, a convergent scheme for FBSDEs with monotone drivers is rarely discussed in the existing literature except for those in \cite{lionnet2015time} and appears to be of great importance for the applicability of FBSDEs to many physical phenomena. It should be mentioned that the $\theta$-schemes suggested in \cite{lionnet2015time} is quite different from ours, on account of the facts that their schemes not only require an approximation of conditional expectations, but also solve a FBSDE whose forward process starts at a particular time with a particular state. As a result, when the initial time or the initial state of the forward process changes, one needs to perform all the computations again, which is clearly very expensive. By contrast, our methods solve for the entire surface of the value function, which is independent of the initial conditions of the forward SDE. Therefore, along with simple numerical schemes for the forward SDE, one can efficiently obtain solutions to FBSDEs with different initial conditions of the forward process, which is of particular importance to mathematical finance, for instance, the evaluation of an option price in terms of  different spot prices of the stocks.

We start by recalling the concept of transposition solutions introduced in  \cite{lu2013well} for BSDEs with general filtration. Then from an important variational formulation of the FBSDEs, depending on the complexity of the drivers, different algorithms for the numerical resolution of FBSDEs are proposed based on an approximation of the corresponding transition semigroups. Another major focus of this work is the convergence of the associated PDE solvers, which will be analyzed via the semigroup theory in virtue of the natural correspondence between second-order elliptic operators and transition semigroups. By first deriving the variation of constants formula of the associated PDEs from the variational formulation of the FBSDEs, we are able to establish the convergence of our explicit scheme for the BSDEs with Lipschitz continuous drivers involving only the state variable.  A rigorous analysis
is also carried out for the hybrid scheme to demonstrate its convergence
for the FBSDEs with maximal monotone drivers of only locally Lipschitz continuity.

We organize this paper as follows. In Section 2, we define some necessary notation and introduce the nonlinear Feynman-–Kac formula, which is of crucial importance to connect the numerical resolution of FBSDEs to the numerical approximation of Cauchy problems.  Then we introduce a variational formulation of FBSDEs in Section 3 motivated by the transposition solutions, from which an explicit scheme, a hybrid scheme and an implicit scheme are derived to solve the FBSDEs. Then we move on to the convergence analysis of our numerical schemes under different assumptions on the drivers. We derive the variation of constants formula and define the mild solutions for a given Cauchy problem in Section 4, which shall help us understand the convergence of our algorithms. Then, a clear convergence analysis of the explicit scheme and the hybrid scheme is separately performed in Section 4 for FBSDEs with Lipschitz continuous drivers and for FBSDEs with maximal monotone drivers. Numerical experiments are presented in Section 5 to illustrate the effectiveness of our algorithms and demonstrate their applications to nonlinear pricing problems for vanilla options with differential interest rates.  

%
\section{Problem formulation and preliminaries}
In this section, we introduce our model FBSDEs and recall the important connection between FBSDEs and PDEs, which is crucial for the
subsequent developments. We start
with some useful notation that is needed frequently in the rest of this work.

We shall write by $T > 0$ the terminal time, and by
$(\Om, \alg{}, P)$ a complete probability space, in which an $m$-dimensional Brownian motion $B_t$ is defined.
We shall denote by $\alg{}=\{\alg{t}\}_{0\le t\le T}$ the natural filtration generated by the Brownian motion augmented by the $P$-null sets and by $\ex$ the usual expectation operator with respect to the measure $P$.

For $(s,x)\in [0,T]\times \R^d$, we denote by $Q^{s,x}$ the probability law of a given Markov process $X_t^{s,x}=\{X_t^{s,x}\}_{t\ge s}$ with initial value $X_s=x$ and by $\ex^{s,x}$ the expectation operator with respect to $Q^{s,x}$. Then we have
$\ex^{s,x}[f(X_t)]=\ex[f(X^{s,x}_t)]$ for all Borel functions $f$ and all time $t\ge s$. If $s=0$, we simply write $X_t^x$, $\ex^x$ and $Q^x$ for $X_t^{0,x}$, $\ex^{0,x}$ and $Q^{0,x}$, respectively.

Furthermore we define three important spaces: $L^2_{\alg{t}}(\Om;\R)$ is the Hilbert space of all $\alg{t}$-measurable $\R$-valued random variable $\xi$ with $\ex[\xi^2]<\infty$, $L^2_{\alg{}}(\Om;L^r(t,s;\R^d))~ (1\le r\le \infty)$ is the Hilbert space of all $\R^d$-valued, $\alg{}$-adapted processes $X(\cdot)$ satisfying $\ex[\norm{X(\cdot)}{L^r(t,s;\R^d)}^2]<\infty$, and $C([a,b];X)$ (resp. $B([a,b];X)$) is the Banach space of all continuous (resp. bounded) functions $u:[a,b]\to X$ for a Banach space $X$. Finally, when no ambiguity arises, we denote by $\norm{\cdot}{}$  the norm  on $L^2(\R^d)$ and by $A\lesssim B$ the relation $A\le CB$, where $C$ is a generic constant independent of time stepsize and mesh size and may take a different value at each occurrence.

Now we are ready to state the problem of our interest. We shall consider the decoupled Markovian FBSDE of the  form:
\begin{equation}\label{eq:BSDE}
\begin{aligncases}
dX_t&=b(X_t)dt+\sigma(X_t)dB_t, &&X_0=x_0,\\
dY_t&=g(t,X_t,Y_t,Z_t)dt+Z_tdB_t, && Y_T=\phi(X_T),
\end{aligncases}
\end{equation}
where the processes $X_t,Y_t,Z_t$ are $\R^d$-valued, $\R$-valued and $\R^m$-valued, respectively. And the driver $g$ and the terminal condition $\phi$ are given deterministic measurable functions. The (path-wise strong) solution of \eqref{eq:BSDE} is a triple of adapted processes $(X_t, Y_t,Z_t)$ which satisfies the equation.

%
 We emphasize that although we shall focus all our discussions in this work
 on the case where the terminal condition of the BSDE in \eqref{eq:BSDE} depends only on the states of $X_t$ at the terminal time,
 our algorithms and analyses can be naturally
 extended to the cases where the terminal conditions may involve the states of $X_t$ at finitely many discrete time points.
 Moreover, to further simplify our presentations, we shall assume that $X_t$ is the time-homogeneous It\^{o} diffusions \cite{oksendal2013stochastic}.
%
\begin{Definition} A time-homogenous Ito diffusion is a stochastic process $X_t(\omega)=X(t,\omega):[0,\infty)\times \Om\to \R^d$
that satisfies a stochastic differential equation of the form
\begin{equation}\label{eq:diffusion}
dX_t=b(X_t)dt+\sigma(X_t)dB_t,\quad t\ge s;\quad X_s=x,
\end{equation}
where  $b:\R^d\to \R^d$ and $\sigma:\R^d\to \R^{d\times m}$ are Lipschitz continuous functions.
\end{Definition}

We will denote the unique solution of \eqref{eq:diffusion} by $\{X_t^{s,x}\}_{t\ge s}$. It is well-known that an Ito diffusion is a sample continuous stochastic process with the property of being time-homogenous, in the sense that
$\{X_{s+h}^{s,x}\}_{h\ge 0}$ and $\{X_h^{0,x}\}_{h\ge 0}$ have the same $P$-distributions, which in turn implies the Markov property. Therefore, if $f$ is a properly defined function, then we can define for any $x\in \R^d$ and $t>0$
the following transition semigroup $S(t)$ associated to $\{X_t^x\}$:
\begin{equation}\label{eq:diffsgp}
S(t)f(x)=\ex^{x}[f(X_t)]=\ex[f(X^x_t)]=\int_\R f(y)p(t,x,y)\,dy,
\end{equation}
where $p(t,x,y)$ denotes the transition density of $X_t$ if it exists. The linear operator $S(t)$ is non-negative,
i.e., $S(t)f\ge 0$ whenever $f\ge 0$ and the semigroup property of $S(t)$ follows from the law of total expectations and the Markov property of $X_t$, i.e., for all $s,t\ge 0$, $x\in \R^d$,
\begin{align*}
S(t+s)f(x)&=\ex[f(X^x_{t+s})]=\ex[\ex[f(X^x_{t+s})\mid \alg{t}](\omega)]=\ex[\ex^{X_t^x(\omega)}[f(X_s)]]\\
&=\ex[S(s)f(X_t^x(\omega))]=S(t)S(s)f(x).
\end{align*}

The rest of this section is devoted to introducing the nonlinear Faynman-Kac formula, which, as we shall see,
not only provides a probabilistic representation of the solution to a class of parabolic PDEs, but also plays an essential role in reducing the numerical resolution of \eqref{eq:BSDE} into the numerical approximation of a backward
semilinear parabolic PDE.

For $(s,x)\in [0,T]\times \R^d$, let $(X_t^{s,x})_{s\le t\le T}$ solve \eqref{eq:diffusion}. 
Consider the semilinear parabolic PDE:
\begin{equation}\label{eq:semilinear}
\begin{aligncases} v_t+\gene v&=g(t,x,v(t,x),\sigma^T \grad_x v(t,x)),&(t,x)\in [0,T)\times \R^d,\\
v(T,x)&=\phi(x), &x\in \R^d,
\end{aligncases}
\end{equation}
where $\gene$ denotes the second order linear differential operator
\begin{equation}\label{eq:elliptic}
\gene=\sum_{i,j=1}^na_{ij}\diff{x_ix_j}{}+\sum_{i=1}^n b_i\diff{x_i}{},\quad a_{ij}=\dfrac{1}{2}[\sigma\sigma^T]_{ij}.
\end{equation}
Using the It\^{o}'s formula, it is straightforward to show that if $v\in C^{1,2}([0,T]\times \R^d)$ solves the above PDE, then $Y^{s,x}_t\coloneqq v(t,X^{s,x}_t)$ and $Z^{s,x}_t\coloneqq \sigma(X^{s,x}_t)^T \grad_x v(t,X^{s,x}_t)$ solves the following BSDE
\begin{equation}\label{eq:BSDE2}
 dY^{s,x}_t=g(t,X^{s,x}_t,Y^{s,x}_t,Z^{s,x}_t)dt+Z^{s,x}_tdB_t,\quad t\in [s,T]\times \R^d;\qq Y_T=\phi(X^{s,x}_T).
 \end{equation}
However, the more interesting result is the converse one, as shown in the next theorem \cite{yong1999stochastic}, the solution to BSDE \eqref{eq:BSDE2} also solves the PDE in some sense.

\begin{Theorem}
Let $b,\sigma$ be uniformly Lipschitz continuous in $x\in \R^d$, $g$ be uniformly Lipschitz continuous in $(y,z)\in \R\times \R^m$ (with respect to $(t,x)\in [0,T]\times \R^d$) and $\phi$ be continuous. Then \eqref{eq:semilinear} admits a unique viscosity solution $v$ that can be represented by 
\begin{equation}\label{bsdetopde}
v(t,x)\equiv Y_t^{t,x},\quad (t,x)\in [0,T]\times \R^d,
\end{equation}
where $(Y_t^{s,x},Z_t^{s,x})$ be the unique adapted solution to \eqref{eq:BSDE2}. In the case where \eqref{eq:semilinear} admits a classical solution, then the solution can be given by \eqref{bsdetopde}.
\end{Theorem}

We remark that in the general case where the terminal condition of the BSDE in \eqref{eq:BSDE} involves the states of $X_t$ at finitely many time points, a PDE representation of the solution to \eqref{eq:BSDE} similar to the system \eqref{eq:semilinear}
can be established by inductively solving a branch of PDEs corresponding to finitely many time intervals \cite{zhang2002discretizing}.

\section{Numerical schemes}\label{section:method}
In this section, we propose three numerical schemes for solving the BSDE \eqref{eq:BSDE} under different assumptions on
the complexity of the driver $g$. All methods are motivated by a weak formulation of \eqref{eq:BSDE} from the viewpoint of transposition solutions, which was introduced in \cite{lu2013well}.
%

We start our discussions by considering the following forward SDEs, whose solutions will work
as test functions in the framework of transposition solutions:

Given $u\in L^2_{\alg{}}(\Om;L^1(s,t;\R))$, $w\in L^2_{\alg{}}(\Om;L^2(s,t;\R^m))$ and $\eta\in L^2_{\alg{s}}(\Om;\R)$, the following  SDE
\begin{equation}\label{eq:SDE}
dS_\tau=u_\tau d\tau+w_\tau dB_\tau,\quad \tau\in [s,t];\qq S_s=\eta,
\end{equation}
 admits a unique solution $S_\tau$ for $\tau\in [s,t]$ in the sense of It\^{o} integral.

Now suppose \eqref{eq:BSDE} admits a strong solution $(X_t,Y_t,Z_t)$, then we can readily derive
for any $t\in [s,T]$ by applying the It\^{o}'s formula to the process $S_tY_t$ and taking the expectations:
\begin{align}\label{eq:variational}
\begin{split}
\ex[S_tY_t]-\ex[\eta Y_s]=&\ex \bigg[\int_s^t S_\tau g(\tau,X_\tau,Y_\tau,Z_\tau)\,d\tau\bigg]+\ex \bigg[\int_s^t u_\tau Y_\tau\,d\tau\bigg]+\ex \bigg[\int_s^t w_\tau^T Z_\tau\,d\tau\bigg],
\end{split}
\end{align}
where we have used the fact that $\ex [\int_s^t \big ( Y_\tau w_\tau+S_\tau Z_\tau\big)\,dB_\tau]=0$.
Motivated by this formula, we come naturally to the following weak formulation for the backward processes of \eqref{eq:BSDE}.

\ms
\textbf{Variational formulation}:
Find a pair $(Y_t,Z_t)\in L^2_{\alg{}}(\Om; L^2(0,T;\R)) \t L^2_{\alg{}}(\Om; L^2(0,T;\R^m)) 
$
such that identity \eqref{eq:variational} holds
for all $s, t\in [0,T]$ with $s\le t$ and $(u,w,\eta)\in L^2_{\alg{}}(\Om;L^1(s,t;\R))\times L^2_{\alg{}}(\Om;L^2(s,t;\R^m))\times L^2_{\alg{s}}(\Om;\R)$.
\ms

A pair of adapted processes $(Y_t,Z_t)$ satisfying the above variational formulation is called a transposition solution to \eqref{eq:BSDE}. It is clear that whenever \eqref{eq:BSDE} admits a strong solution, it coincides with the transposition solution. We refer readers to  \cite{lu2013well} for the well-posedness of \eqref{eq:BSDE} in the sense of transposition solutions, which was established without the Martingale representation theorem. 
\textcolor{black}{We remark that transposition solutions generalize the duality relationship between linear BSDEs and 
SDEs (see, e.g. \cite{yong1999stochastic}) to nonlinear BSDEs, and consequently enables our 
construction of numerical schemes for BSDEs with general filtration, which will not be discussed in this work.}

We now propose some numerical schemes for solving the FBSDEs \eqref{eq:BSDE}.
It is important for us to point out that
we shall not intend to directly solve the PDE \eqref{eq:semilinear},
but the nonlinear Faynman-Kac formula that connects the solutions
to FBSDEs and PDEs
is the principal idea that suggests us to construct our numerical schemes
under the novel framework of transposition solutions.

Let $0=t_0<t_1<\cdots<t_N=T$ be a time partition of $[0,T]$ with time stepsize $\Del_{k+1}=t_{k+1}-t_k$.
Motivated by the nonlinear Faynman-Kac formula, we can write $$Y_{t_k}=v(t_k,X_{t_k}),\quad Z_{t_k}=\sigma(X_{t_k})^T\nabla_x v(t_k,X_{t_k}),\quad k=0,\cdots,N,$$
for some function $v$. We now select
a set of orthonormal basis $\{h_i\}_{i=1}^\infty$ in $L^2(\R^d)$, then
approximate $(Y_{t_k},Z_{t_k})$ for each discrete time point $t_k$ by
\begin{equation}\label{eq:YZ}
\hat{Y}_{t_k}=\sum_{i=1}^\infty\alpha_k^i h_i(X_{t_k}),\quad \hat{Z}_{t_k}=\sum_{i=1}^\infty\beta_k^i h_i(X_{t_k}).
\end{equation}
That is, we approximate $v(t_k,x)$ and $\sigma(x)^T\nabla_x v(t_k,x)$ respectively by
\begin{equation}\label{eq:vw}
v^k=\sum_{i=1}^\infty\alpha_k^i h_i(x),\quad w^k=\sum_{i=1}^\infty\beta_k^i h_i(x)\,.
\end{equation}

Suppose $p_{{X_{t_k}} }$, $k=0,1,\cdots,N$, are the density functions of $X_{t_k}$ with
$p_{X_{t_k}}>0$ almost everywhere  and
$p^{-1}_{X_{t_k}}$ are the reciprocals of $p_{X_{t_k}}$ 
\textcolor{black}{(see e.g. \cite{fournier2010absolute} for the existence of such density functions for It\^{o} diffusions)}.
Then the orthonormality of $\{h_i\}_{i=1}^\infty$ implies
\begin{equation}\label{eq:ortho}
\ex[p^{-1}_{X_{t_k}}(X_{t_k})h_i(X_{t_k})h_j(X_{t_k})]=\delta_{ij}~~\q \fa i,j \ge 1, k=0,1,\cdots,N\,.
\end{equation}

Next we discuss how to compute all the coefficients $\{\alpha_k^i\}$ and $\{\beta_k^i\}$ in \eqref{eq:YZ}.
First for $t_N=T$,
by representing the terminal condition of \eqref{eq:BSDE} in the form of \eqref{eq:YZ},
we derive from \eqref{eq:ortho} that
\begin{equation}\label{eq:alphatN}
\alpha_N^i=\ex[p^{-1}_{X_T}(X_T)\phi(X_T)h_i(X_T)]=\int_{\R^d} \phi(x)h_i(x)\, dx, \quad i\ge 1.
\end{equation}

For each $t_k$ with $k=N-1,\cdots, 1,0$, we consider \eqref{eq:SDE} on $[t_k,t_{k+1}]$.
Selecting $u=0,w=0,\eta=h_i(X_{t_k})p^{-1}_{X_{t_k}}(X_{t_k})$ in \eqref{eq:SDE}, we see
$S_\tau=h_i(X_{t_k})p^{-1}_{X_{t_k}}(X_{t_k})$ for $\tau\in [t_k,t_{k+1}]$.
Then we can deduce from \eqref{eq:variational} and \eqref{eq:ortho} that for $i\ge 1$,
\begin{eqnarray}
\alpha_k^i&=&\ex[p^{-1}_{X_{t_k}}(X_{t_k})h_i(X_{t_k})Y_{t_{k+1}}]-\ex\bigg[\int_{t_k}^{t_{k+1}}p^{-1}_{X_{t_k}}(X_{t_k})h_i(X_{t_k})g(t_\tau,X_\tau,Y_\tau,Z_\tau)\,d\tau\bigg] \label{eq:alpha}\\
&=&\sum_{j=1}^\infty\ex[p^{-1}_{X_{t_k}}(X_{t_k})h_i(X_{t_k})h_j(X_{t_{k+1}})]\alpha^j_{k+1}-\ex\bigg[\int_{t_k}^{t_{k+1}}p^{-1}_{X_{t_k}}(X_{t_k})h_i(X_{t_k})g(t_\tau,X_\tau,Y_\tau,Z_\tau)\,d\tau\bigg]. \nb
\end{eqnarray}
%

The above derivation holds for any orthonormal basis $\{h_i\}$. Now we choose $\{h_i\}$ to be the indicator functions
to further simplify the expression \eqref{eq:alpha}. 
Let $\Z$ be the set of all integers,
$\mathcal{R}\coloneqq\{x_j\mid x_j\in \R, x_j<x_{j+1},j\in \Z, \lim_{j\to \pm\infty}x_j=\pm\infty\}$
be a spatial partition of the real axis $\R$, and $\mathcal{R}^d=\prod_{j=1}^d\mathcal{R}_j$ be a  partition of the Euclidean space $\R^d$. For each $i=(i_1,\cdots,i_d)\in \Z^d$, let $I_i=\prod_{j=1}^d(x_{i_j},x_{i_j+1}]$ with size $|I_i|=\prod_{j=1}^d (x_{i_j+1}-x_{i_j})$.
Then we choose $h_i$ to be the indicator function
$h_i(x)=1_{I_i}(x)/{\sqrt{|I_i|}}$ for $i\in \Z^d$. It is clear that $\{h_i\}_{i\in \Z^d}$ is orthonormal in $L^2(\R^d)$. 
We remark that although the idea of representing the numerical solutions in terms of the piecewise smooth basis is similar to  discontinuous Galerkin methods for solving PDEs \cite{riviere2008discontinuous}, our scheme is essentially different since no  stabilization terms are introduced to enforce the weak continuities of the numerical solutions on the common faces between 
neighboring elements.

Next we give the detailed update of the coefficients
$\{\alpha_k^i\}$  for each $k=N-1,\cdots, 1,0$. Define
$$P^k_{ij}=\ex[p^{-1}_{X_{t_k}}(X_{t_k})h_i(X_{t_k})h_j(X_{t_{k+1}})] ~~\q \fa i,j\in \Z^d, $$
then the Markov property of $X_t$ implies that $P^k=(P^k_{ij})_{i,j\in\Z}$ is a discrete version of the transition semigroup \eqref{eq:diffsgp} in the sense that for any $i,j\in \Z^d$, the following represention formula holds
\begin{equation}\label{eq:Pij}
P^k_{ij}=\inprod{S(\Del_{k+1})h_j}{h_i}_{L^2(\R^d)},\quad k=0,1,\cdots, N-1.
\end{equation}
In fact, for any $k=0,\cdots,N-1$, we have
\begin{align*}
P^k_{ij}&=\ex[p^{-1}_{X_{t_k}}(X_{t_k})h_i(X_{t_k})h_j(X_{t_{k+1}})]=\ex\big[\ex[p^{-1}_{X_{t_k}}(X_{t_k})h_i(X_{t_k})h_j(X_{t_{k+1}})\mid \alg{t_k}]\big]\\
&=\ex\big[p^{-1}_{X_{t_k}}(X_{t_k})h_i(X_{t_k})\ex[h_j(X_{t_{k+1}})\mid \alg{t_k}]\big]=\ex\big[p^{-1}_{X_{t_k}}(X_{t_k})h_i(X_{t_k})\ex^{X_t}[h_j(X_{\Del_{k+1}})]\big]\\
&=\int_{\R^d} h_i(x)\ex^{x}[h_j(X_{\Del_{k+1}})]\, dx=\inprod{S(\Del_{k+1})h_j}{h_i}_{L^2(\R^d)}.
\end{align*}
Furthermore, suppose the process $X_t$ has a transition density $p(t,x,y)$, then \eqref{eq:diffsgp} and \eqref{eq:Pij} yield
\bb\label{eq:Pijdensity}
P^k_{ij}=\frac{1}{\sqrt{|I_i|\t |I_j|}}\iint_{I_i\times I_j}p(\Del_{k+1},x,y)\,dxdy ~~\quad \fa i,j\in \Z^d, k=0,1,\cdots, N-1.
\ee

Finally, based on whether the driver $g$ involves the component $Z$, the integral in \eqref{eq:alpha} is approximated by different quadrature rules, which results in different algorithms.

\ms
\textbf{Case 1}: the driver $g$ is independent of the component $Z$. In this case,
we can derive an explicit scheme for solving BSDE \eqref{eq:BSDE} and PDE \eqref{eq:semilinear}.

    Approximating the integral in  \eqref{eq:alpha} by the right endpoint rule yields that
    \begin{align*}
    \ex\bigg[\int_{t_k}^{t_{k+1}}p^{-1}_{X_{t_k}}(X_{t_k})h_i(X_{t_k})&g(t_\tau,X_\tau,Y_\tau)\,d\tau\bigg]
    \approx \ex[p^{-1}_{X_{t_k}}(X_{t_k})h_i(X_{t_k})g(t_{k+1},X_{t_{k+1}},\hat{Y}_{t_{k+1}})]\Del_{k+1}\\
    &= \ex[p^{-1}_{X_{t_k}}(X_{t_k})h_i(X_{t_k})g(t_{k+1},X_{t_{k+1}},\sum_{j\in \Z^d}\alpha_{k+1}^j h_j(X_{t_{k+1}}))]\Del_{k+1}.
    \end{align*}

    Suppose $X_t$ has the transition density $p(t,x,y)$ and $X_{t_k}$ has the density $p_{X_{t_k}}$, then
    using the above expression and the joint density $p_{X_{t_k}}(x)p(\Del_{k+1},x,y)$ of $(X_{t_k},X_{t_{k+1}})$
    we deduce that
    \begin{align}\label{eq:gexp}
    \begin{split}
    &\ex[\int_{t_k}^{t_{k+1}}p^{-1}_{X_{t_k}}(X_{t_k})h_i(X_{t_k})g(t_\tau,X_\tau,Y_\tau)\,d\tau]\\
    \approx& \frac{1}{\sqrt{|I_i|}}\bigg(\iint_{\R^{2d}} p_{X_{t_k}}(x)p(\Del_{k+1},x,s) p^{-1}_{X_{t_k}}(x)1_{I_i}(x)g(t_{k+1},s,\sum_{j\in \Z^d}\frac{\alpha_{k+1}^j}{\sqrt{|I_j|}} 1_{I_j}(s))\,dsdx\bigg)\Del_{k+1}\\
    =& \frac{\Del_{k+1}}{\sqrt{|I_i|}}\sum_{j\in \Z^d} \big(\iint_{I_i\times I_j} p(\Del_{k+1},x,s) g(t_{k+1},s,\frac{\alpha_{k+1}^j}{\sqrt{|I_j|}} )\,dsdx\big)\\
    \approx &\Del_{k+1}\sum_{j\in \Z^d} \bigg(\sqrt{|I_j|}P^k_{ij}\bigg)\bigg(\frac{1}{|I_j|}\int_{I_j}g(t_{k+1},s,\frac{\alpha_{k+1}^j}{\sqrt{|I_j|}})\, ds\bigg),
    \end{split}
    \end{align}
    where we have used the representation formula \eqref{eq:Pijdensity} for $P^k_{ij}$ in the last approximation.

    Therefore, using \eqref{eq:alpha},  \eqref{eq:Pij} and \eqref{eq:gexp}, we come to the following scheme:
    \begin{equation}\label{eq:alphatkexp}
    \alpha_k^i=\sum_{j\in \Z^d} P^k_{ij}\big(\alpha_{k+1}^j-\Del_{k+1}\f{1}{\sqrt{|I_j|}}\int_{I_j}g(t_{k+1},s,\frac{\alpha_{k+1}^j}{\sqrt{|I_j|}})\, ds\big) ~~\quad \fa i\in \Z^d, k=0,1,\cdots, N-1,
    \end{equation}
    On the other hand, we can update the coefficients $\b^k=\{\b_k^i\}$  of $w^k$ explicitly by
    \begin{equation}\label{eq:beta}
    \beta^k=\sigma^TD\alpha^{k+1},\q k=0,1,\cdots, N-1,
    \end{equation}
    where $D$ is a difference scheme of $\nabla_x$, e.g.,
    the central difference or upwinding scheme.

    Summarizing the above discussions, we come to the following explicit scheme.

    \ms
    \textbf{Algorithm 1}.
    For $k=N,N-1, \cdots,0$, compute the coefficients $\alpha^k,\beta^k$ by the formulas \eqref{eq:alphatN},

    \eqref{eq:alphatkexp} and \eqref{eq:beta}. Then compute $(\hat{Y}_{t_k}, \hat{Z}_{t_k})$ (resp. $(v^k,w^k)$) as in \eqref{eq:YZ} (resp. \eqref{eq:vw}).
\ms

{We remark that 
the above scheme was derived under the concept of transposition solutions. 
As it is seen later, this helps us interpret it as an operator-splitting approximation 
of the variation of constants formula \eqref{eq:inteq}, as well as 
derive high-order schemes for general BSDEs.}

\ms
\textbf{Case 2}: the driver $g$ involves the component $Z$.
We shall propose a hybrid scheme and an implicit scheme for solving BSDE \eqref{eq:BSDE} and PDE \eqref{eq:semilinear}.

We first derive the hybrid scheme by using the left endpoint quadrature rule in  \eqref{eq:alpha} to get
    \begin{align}\label{eq:gimp}
    &\ex[\int_{t_k}^{t_{k+1}}p^{-1}_{X_{t_k}}(X_{t_k})h_i(X_{t_k})g(t_\tau,X_\tau,Y_\tau,Z_\tau)\,d\tau]
    \approx \ex[p^{-1}_{X_{t_k}}(X_{t_k})h_i(X_{t_k})g(t_{k},X_{t_{k}},\hat{Y}_{t_{k}},\hat{Z}_{t_{k}})]\Del_{k+1}\nb\\
    =& \ex[p^{-1}_{X_{t_k}}(X_{t_k})h_i(X_{t_k})g(t_{k},X_{t_{k}},\sum_{j\in \Z^d}\frac{\alpha_{k}^j}{\sqrt{|I_j|}} 1_{I_j}(X_{t_{k}}),
    \sum_{j\in \Z^d}\frac{\beta_{k}^j}{\sqrt{|I_j|}} 1_{I_j}(X_{t_{k}}))]\Del_{k+1}\nb\\
    =&\frac{\Del_{k+1}}{\sqrt{|I_i|}} \int_{I_i}g(t_{k},s,\frac{\alpha_{k}^i}{\sqrt{|I_i|}},
    \frac{\beta_{k}^i}{\sqrt{|I_i|}} )\, ds.
    \end{align}

    We can compute the coefficients $\beta^k$ explicitly by \eqref{eq:beta}. This, along with
    \eqref{eq:alpha} and \eqref{eq:Pij}, yields a scheme which solves the following equation
    for the coefficient $\alpha^{k}$ (with $P_{ij}$ from \eqref{eq:Pij}):
    \begin{equation}\label{eq:alphatkimp}
    \alpha_k^i=\bigg(\sum_{j\in \Z^d} P^k_{ij}\alpha_{k+1}^j\bigg)-\frac{\Del_{k+1}}{\sqrt{|I_i|}} \int_{I_i}g(t_{k},s,\frac{\alpha_{k}^i}{\sqrt{|I_i|}},
    \frac{\beta_{k}^i}{\sqrt{|I_i|}} )\, ds
        \quad \fa i\in \Z^d, k=0,\cdots, N-1\,.
    \end{equation}

    The above discussions lead to the hybrid scheme that is implicit in $Y$ and explicit in $Z$.

    \ms
    \textbf{Algorithm 2}.
    For $k=N,N-1, \cdots,0$, compute the coefficients $\alpha^k,\beta^k$ by the formulas \eqref{eq:alphatN},

    \eqref{eq:alphatkimp} and \eqref{eq:beta}. Then compute $(\hat{Y}_{t_k}, \hat{Z}_{t_k})$ (resp. $(v^k,w^k)$) as in \eqref{eq:YZ} (resp. \eqref{eq:vw}).

    \ms
    Alternatively, replacing $\a^{k+1}$ in \eqref{eq:beta} by $\a^{k}$, along with \eqref{eq:alphatkimp}, suggests us 
    to compute the coefficients $(\a^k, \beta^k)$  for $k=0,\cdots, N-1$ as follows: $\beta^k=\sigma^T D\alpha^k$ and 
    \bb\label{eq:alphabetafullimp}
    \alpha_k^i=\bigg(\sum_{j\in \Z^d} P^k_{ij}\alpha_{k+1}^j\bigg)-\frac{\Del_{k+1}}{\sqrt{|I_i|}} \int_{I_i}g(t_{k},s,\frac{\alpha_{k}^i}{\sqrt{|I_i|}},
    \frac{\beta_{k}^i}{\sqrt{|I_i|}} )\, ds,\q i\in \Z^d\,.
    \ee
    This leads to the following implicit scheme.

    \ms
    \textbf{Algorithm 3}.
    For $k=N,N-1, \cdots,0$, compute the coefficients $\alpha^k,\beta^k$ by the formulas \eqref{eq:alphatN}

    and \eqref{eq:alphabetafullimp}. Then compute $(\hat{Y}_{t_k}, \hat{Z}_{t_k})$ (resp. $(v^k,w^k)$) as in \eqref{eq:YZ} (resp. \eqref{eq:vw}).

    \ms
    We end this section with an important remark about some differences and similarities among all these algorithms proposed above.
    Based on the structure of the driver $g$, we shall employ different schemes to achieve the better balance between the numerical stability and the computational costs.
    For FBSDEs whose drivers are Lipschitz in $Y$ but independent of $Z$, the explicit Algorithm 1 shall be adopted. The hybrid Algorithm 2 is for FBSDEs whose drivers are Lipschitz in $Z$ with a Lipschitz or monotone dependence on $Y$, while the implicit Algorithm 3 can be used for the general case, in particular for FBSDEs whose drivers are monotone in $Y$  and locally Lipschitz and of polynomial growth in $Z$. In the next section, we shall perform a careful convergence analysis
for the first two algorithms.

    Apart from the aforementioned differences, as we will see from our analyses in Section \ref{section:conv}, the constructions of our above schemes can be interpreted as a combination of  difference approximations in time and operator-splitting methods for the variation of constants formula \eqref{eq:inteq}. This interpretation enables us to derive systematically high-order schemes based on high-order splitting of the variation of constants formula. Moreover, in all these numerical algorithms we use $P^k$ to represent the semigroup $S(\Del_{k+1})$, which in general refers to the semigroup associated with a Markov process; see e.g. \cite{varadhan2007stochastic}. An empirical point of view is that for a Markov process $X_t$ with state space $E$ and a partition $\{E_k\}_{k=1}^\infty$ of $E$, i.e., $\bigcup_{k=1}^\infty E_k=E$, we have the transition probability $P^k_{ij}=P(X_{\Delta_{k+1}} \in E_j|X_0 \in E_i)$.
 
\color{black}   
We finally remark that in this work we focus our attention on the numerical methods for the FBSDEs whose variation of constants formulas admit differentiable mild solutions; see \cite{ito2002evolution} and \cite{yong1999stochastic} for sufficient conditions. This motivates us to propose difference schemes to approximate the martingale process $Z$ based on the the nonlinear Faynman-Kac formula. For general FBSDEs, one may construct numerical schemes  based on Malliavin  Monte Carlo weights to approximate  the $Z$ process as suggested in \cite{bouchard2004discrete} and \cite{zhang2004numerical}.
\color{black}

\section{Convergence analysis}\label{section:conv}
In this section, we recall the concept of a mild solution to the PDE \eqref{eq:semilinear} and establish the convergence of our numerical solution $v^k_h$ to the mild solution $v$. We will analyze the convergence of our Algorithms 1 and 2 introduced in Section \ref{section:method}, for the case where the driver $g$ is independent of the component $Z$, and satisfies a uniformly Lipschitz continuity or maximal monotonicity in the component $Y$, respectively. For both cases, we shall perform the convergence analysis first for the temporal discrete scheme and then for the fully discrete scheme.

We start the discussion by deriving a variation of constants formula for \eqref{eq:BSDE} from the variational identity \eqref{eq:variational}. For the Markovian case, i.e., $Y_T=\phi(X_T)$, we have $Y_t=v(t,X_t)$ and $Z_t=\b(t,X_t)$ for $t\in [0,T)$. Let $h$ be an arbitrary bounded integrable function on $\R^d$ and $p_{X_t}$ be the probability density function of the Markov process $X_t$.

Choosing
$u=0$, $\eta=0$ and $w= p^{-1}_{X_t}(X_t)h_i(X_t)$ in \eqref{eq:SDE},
we have
$S_\tau =p^{-1}_{X_t}(X_t)h(X_t)(B_\tau-B_t)$, and then derive from
\eqref{eq:variational} that
\begin{align*}
\b(t,x)h(x)&\approx \ex[p^{-1}_{X_t}(X_t)h(X_t)v(\tau,X_\tau)(B_\tau-B_t)]\\
&\approx \ex[p^{-1}_{X_t}(X_t)h(X_t)\p_x v(\tau,X_\tau)(dX_\tau)(dB_\tau)]=\sigma(x)^T\nabla_x v(t,x)h(x)+o(|\tau-t|),
\end{align*}
consequently we see $\beta(t,x)=\sigma(x)^T\nabla_x v(t,x)$ and $Z_t=\sigma(X_t)^T \nabla_x v(t,X)$ for $t\in [0,T)$.

Similarly, we get $S_\tau =p^{-1}_{X_t}(X_t)h(X_t)$ by choosing $u=w=0$
and $\eta= p^{-1}_{X_t}(X_t)h_i(X_t)$) in \eqref{eq:SDE}, then we can derive from \eqref{eq:variational} that
\begin{align*}
v(t,x)h(x)&=\ex[p^{-1}_{X_t}(X_t)h(X_t)Y_s]-\ex[\int_t^s p^{-1}_{X_t}(X_t)h_i(X_t)g(t_\tau,X_\tau,Y_\tau,Z_\tau)\,d\tau]\\
&=\bigg(S(s-t)v(s,x)-\int_t^s S(\tau-t)g(\tau,x,v(\tau,x),\sigma(x)^T\nabla_x v(\tau,x)))\,d\tau\bigg) h(x),
\end{align*}
where $\{S(t)\}_{t\ge 0}$ is the transition semigroup associated with $X_t$ defined as in \eqref{eq:diffsgp}. Therefore, we infer $v$ satisfies the integral equation: ~$v(T,x)=\phi(x)$ for $x\in \R^d$, and for all $t,s\in [0,T),s\le t$,
\begin{equation}\label{eq:inteq}
 v(s,x)=S(t-s)v(t,x)-\int_s^t S(\tau-s)g(\tau,x,v(\tau,x),\sigma(x)^T \nabla_x v(\tau,x))\,d\tau, 
\end{equation}

The integral equation \eqref{eq:inteq} is  called a variation of constants formula for \eqref{eq:BSDE} and a function $v$ satisfying \eqref{eq:inteq} is said to be a mild solution to \eqref{eq:semilinear}. Furthermore, we refer to \cite{pazy2012semigroups} for the sufficient conditions under which a mild solution is a classical solution to \eqref{eq:semilinear}, and to \cite{ito2002evolution,ito2008feedback} for the applications of mild solutions on numerical analysis and optimal control.
We remark that similar integral representation of the solution to \eqref{eq:BSDE} has been established in \cite{ma2002representation,pardoux2014stochastic} through a probabilistic argument. 

In the rest of this section, we assume the existence and uniqueness of the mild solution $v$ and establish the convergence of our numerical solution to $v$. For the sake of simplicity, we only consider the case with a uniform partition of $[0,T]$ (resp. $\R^d$) with time stepsize $\Del t$ (resp. with mesh size $h$), but the results can be easily generalized to a non-uniform regular partition.

The following discrete Gronwall inequality will be used frequently in our subsequent analysis.
\begin{Lemma}
If three nonnegative sequences  $\{a_i\}$, $\{b_i\}$ and $\{c_i\}$ satisfy that
$a_{i-1}\le (1 + c_i)a_i + b_i$ for $i \ge 1$, then it holds for $n\ge 0$ that
$$\max_{0\le i\le n}a_n\le \exp \big(\sum_{i=1}^n c_i\big)\big(a_n+\sum_{i=1}^n b_i\big).$$
\end{Lemma}

\subsection{Convergence  for a uniformly Lipschitz continuous driver $g$}\label{sec:alog1lip}
In this section, we demonstrate
that our numerical solution from Algorithm 1 provides a good approximation to the mild solution $v(t)$
for a Lipschitz continuous driver $g$. We shall first perform the convergence analysis for  the semi-discrete scheme
and then for the fully discrete scheme.

\subsubsection{Convergence of the semi-discrete scheme}
We shall carry our analysis under the following assumptions:

\newpage
\textbf{Assumption 1}.
\begin{enumerate}[{(a)}]
\item The family $\{S(t)\}_{0\le t\le T}$ of operators  is a strongly continuous semigroup of bounded linear operators on $L^2(\R^d)$,
i.e.,  $\lim_{t\to 0}\norm{S(t)f-f}{}=0$ for all $f\in L^2(\R^d).$
\item The terminal condition $\phi\in L^2(\R)$ and $h(\cdot)\coloneqq g(0,\cdot,0)\in L^2(\R^d)$.
\item There exists a unique mild solution $v\in C([0,T];L^2(\R^d))$ to \eqref{eq:inteq}.
\end{enumerate}
Also, we consider the driver $g$ to be uniformly Lipschitz continuous.

\ms
\textbf{Assumption 2}.
 There exists a constant $L>0$ such that
    $$\norm{g(t_1,\cdot,u(\cdot))-g(t_2,\cdot,v(\cdot))}{}\le L(|t_1-t_2|+\norm{u-v}{})~~\quad \fa (t_1,u),(t_2,v)\in [0,T]\times L^2(\R^d).$$

For our subsequent analysis, we introduce a semi-discrete version of
Algorithm 1 (Section \ref{section:method}):
%
\begin{equation}\label{eq:timediscretelip}
 v^n=S(\size{t})T(\size{t})v^{n+1}, \q n=0,\cdots,N-1; \q
 v^N=\phi,
\end{equation}
where operator $T(\size{t})$ is defined explicitly for any given $v\in L^2(\R^d)$:
\begin{equation}\label{eq:Tlip}
T(\size{t})v(x)\coloneqq v(x)-\size{t}g(t_{n+1},x,v(x))\q   \textnormal{for \it{a.e.} $x\in \R^d$}.
\end{equation}

For each $n=0,\cdots, N-1$, the local truncation error of the scheme \eqref{eq:timediscretelip}
at $t_{n+1}$ is defined by
\begin{equation}\label{eq:localerr}
R^{n+1}_{\size{t}}(x)=v(t_n,x)-S(\size{t})T(\size{t})v(t_{n+1},x)\q  \textnormal{for \it{a.e.} $x\in \R^d$}.
\end{equation}

The following lemma illustrates the Lipschitz property of operator $T(\size{t})$, which can be verified
readily by using the Lipschitz continuity of $g$.
\begin{Lemma}\label{thm:T1}
Under Assumption 2, for any given $\size{t}>0$,
the operator $T(\size{t})$ is a Lipschitz continuous operator from $L^2(\R^d)$ into $L^2(\R^d)$ in the sense that
$$\norm{T(\size{t})v-T(\size{t})w}{}\le (1+L \size{t})\norm{v-w}{} ~~\quad \fa v,w\in L^2(\R^d).$$
\end{Lemma}
%
%

Now we are ready to conclude the consistency and convergence of the scheme \eqref{eq:timediscretelip}.

\begin{Theorem}\label{thm:timelip} Under Assumptions 1 and 2,
the scheme \eqref{eq:timediscretelip} is consistent in the sense that
\begin{equation}\label{eq:timeconsistency}
\sum_{n=0}^{N-1}\norm{R^{n+1}_{\size{t}}}{}\to 0 \q  \textnormal{\it{as} $N\to \infty$}.
\end{equation}
Moreover, the solution $v^n$ of the scheme converges to the mild solution $v$ in the sense that
$$\max_{0\le n\le N}\norm{v(t_n)-v^n}{}\to 0\q  \textnormal{\it{as} $N\to \infty$}.$$
\end{Theorem}
\begin{proof}

Rearranging the terms in \eqref{eq:localerr} leads to
\begin{equation}\label{eq:trunerr}
v(t_n,x)=S(\size{t})T(\size{t})v(t_{n+1},x)+R^{n+1}_{\size{t}}(x)\q  \textnormal{for \it{a.e.} $x\in \R^d$}.
\end{equation}
Then for almost every $x\in \R^d$, subtracting  \eqref{eq:trunerr} from the scheme \eqref{eq:timediscretelip}, we obtain
\begin{align*}
|v(t_n,x)-v^n(x)|&=|S(\size{t})T(\size{t})v(t_{n+1},x)-S(\size{t})T(\size{t})v^{n+1}(x)+R^{n+1}_{\size{t}}(x)|\\
&\le |S(\size{t})T(\size{t})v(t_{n+1},x)-S(\size{t})T(\size{t})v^{n+1}(x)|+|R^{n+1}_{\size{t}}(x)|.
\end{align*}
Taking $L^2$-norm on both sides of the above expression and using the boundedness of $S(t)$ and Lemma \ref{thm:T1}, we have the following estimate of the numerical error:
\begin{equation}\label{eq:err}
\norm{v(t_n)-v^n}{}\lesssim (1+\size{t})\norm{v(t_{n+1})-v^{n+1}}{}+\norm{R^{n+1}_{\size{t}}}{},\q n=0,1\cdots, N-1,
\end{equation}
from which, the discrete Gronwall inequality and the condition $v(T)=v^N$ it follows that
\begin{align*}
\max_{0\le n\le N}\norm{v(t_n)-v^n}{}\lesssim \norm{v(T)-v^N}{}+\sum_{n=0}^{N-1}\norm{R^{n+1}_{\size{t}}}{}=\sum_{n=0}^{N-1}\norm{R^{n+1}_{\size{t}}}{}.
\end{align*}

Now we can easily see that it suffices to establish the consistency \eqref{eq:timeconsistency}
in order to conclude the convergence of \eqref{eq:timediscretelip}.
From \eqref{eq:inteq}, \eqref{eq:timediscretelip} and \eqref{eq:localerr}, we deduce for almost every $x\in\R^d$ that
\begin{align*}
R^{n+1}_{\size{t}}(x)=&\size{t}S(\size{t})g(t_{n+1},x,v(t_{n+1},x))-\int_{t_n}^{t_{n+1}}S(\tau-t_n)g(\tau,x,v(\tau,x))\,d\tau\\
=&\int_{t_n}^{t_{n+1}}S(\size{t})\big(g(t_{n+1},x,v(t_{n+1},x))-g(\tau,x,v(\tau,x))\big)\,d\tau\\
&+\int_{t_n}^{t_{n+1}}\big(S(\size{t})-S(\tau-t_n)\big)g(\tau,x,v(\tau,x))\,d\tau\\
\le & \int_{t_n}^{t_{n+1}}S(\size{t})\big(|g(t_{n+1},x,v(t_{n+1},x))-g(\tau,x,v(\tau,x))|\big)\,d\tau\\
&+\int_{t_n}^{t_{n+1}}\big(S(\size{t})-S(\tau-t_n)\big)|g(\tau,x,v(\tau,x))|\,d\tau\\
\coloneqq & A_n(x)+B_n(x).
\end{align*}

Using the inequality for Bochner integral, the boundedness of $S(t)$ and the Lipschitz continuity  of $g$, we obtain the following estimates:
\begin{align*}
\norm{A_n}{}&\le \int_{t_n}^{t_{n+1}}\norm{S(\size{t})\big|g(t_{n+1},\cdot,v(t_{n+1},\cdot))-g(\tau,\cdot,v(\tau,\cdot))\big|}{}\,d\tau\\
&\lesssim\int_{t_n}^{t_{n+1}} |t_{n+1}-\tau|+\norm{v(t_{n+1})-v(\tau)}{}\,d\tau,\\
\norm{B_n}{}&\le\int_{t_n}^{t_{n+1}}\norm{\big(S(\size{t})-S(\tau-t_n)\big)|g(\tau,\cdot,v(\tau,\cdot))|}{}\,d\tau.
\end{align*}

Since $\{S(t)\}_{t\ge 0}$ is a strongly continuous semigroup on $L^2(\R^d)$ and the mild solution $v\in C([0,T];L^2(\R^d))$, for any given $\epsilon>0$, there exists $0<t_0<\epsilon$ independent of $N$, such that for any $\size{t}<t_0$, we have
$$\norm{v(t_{n+1})-v(\tau)}{}< \epsilon,\q \fa \tau\in [t_n,t_{n+1}].$$
We can immediately obtain from the above estimates that for any $\size{t}<t_0$,
\begin{align*}
\sum_{n=0}^{N-1}\norm{R^{n+1}_{\size{t}}}{}\le \sum_{n=0}^{N-1}\big(\norm{A_n}{}+\norm{B_n}{}\big)
\lesssim \epsilon T+ \sum_{n=0}^{N-1}\int_{t_n}^{t_{n+1}}\norm{\big(S(\size{t})-S(\tau-t_n)\big)|g(\tau,\cdot,v(\tau,\cdot))|}{}\,d\tau,
\end{align*}
from which it suffices for us to establish
\begin{equation}\label{eq:flip}
 \sum_{n=0}^{N-1} \int_{t_n}^{t_{n+1}}\norm{\big(S(\size{t})-S(\tau-t_n)\big)|g(\tau,\cdot,v(\tau,\cdot))|}{}\,d\tau=\int_0^T f(\tau)\, d\tau\to 0 \q \textrm{as}\; N\to \infty,
 \end{equation}
 where $f(\tau)\coloneqq \sum_{n=0}^{N-1} \norm{\big(S(\size{t})-S(\tau-t_n)\big)|g(\tau,\cdot,v(\tau,\cdot))|}{}1_{[t_n,t_{n+1})}(\tau)$ for $\tau\in [0,T]$.
 In fact, As $v\in C([0,T];L^2(\R^d))$ and $S(t)$ are bounded, we obtain for each $\tau\in[0,T]$ that
 \begin{align*}
 f(\tau)&\lesssim \norm{g(\tau,\cdot,v(\tau,\cdot))}{}\le  \norm{g(\tau,\cdot,v(\tau,\cdot))-g(0,\cdot,0)}{}+\norm{g(0,\cdot,0)}{}\le L( \tau+\norm{v(\tau)}{})+\norm{g(0,\cdot,0)}{}.
 \end{align*}
 Moreover, for any fixed $\tau\in [0,T]$, we know $\tau\in [t_n,t_{n+1}]$ for some $n$ and $g(\tau,\cdot,v(\tau,\cdot)\in L^2(\R^d)$. Also we note that $t_{n+1}-\tau\le \size{t}$,
 hence it  follows from the strong continuity of $S(t)$ that $f(\tau)\to 0$ as $N\to \infty$. Now the Dominated Convergence Theorem concludes \eqref{eq:flip} directly.
\end{proof}

\subsubsection{Convergence of the fully discrete scheme}

Now we proceed to investigate the convergence of our fully discrete Algorithm 1 to the mild solution $v$ satisfying the variation of constants formula \eqref{eq:inteq}. Given a uniform spatial partition $\{I_j\}_{j\in\Z^d}$ of  $\R^d$ with mesh size $h$, 
we define the projection operator $P_h$ mapping $f\in L^2(\R^d)$ to the piecewise constant function $P_hf$ by
$$P_hf(x)=\sum_{j\in \Z^d} \alpha_j1_{I_j}(x),\q  x\in \R^d, \q \textrm{with}\q \alpha_j=\alpha_j(f)=\frac{1}{|I_j|}\int_{I_j}f(x)\,dx ~~\q \fa j\in \Z^d.$$

Since $\{I_j\}_{j\in\Z^d}$ are disjoint hypercubes, $P_hf$ is well-defined. And
it is easy to verify by the definition that
%
the projection operators $\{P_h\}_{h>0}$ are positive and satisfy
$$\norm{P_h}{\gene(L^2(\R^d))}\le 1 ~~\q \fa h>0.$$

With this projection operator, we now define two operators for any given $h>0$ and $\size{t}>0$:
\begin{equation}\label{eq:discreteop}
S_h(\size{t})=P_hS(\size{t}),\quad T_h(\size{t})=P_hT(\size{t})P_h,
\end{equation}
where $S(\size{t})$ is the transition semigroup associated with $X_t$ defined as in \eqref{eq:diffsgp} and $T(\size{t})$ is the explicit Euler scheme operator defined as in \eqref{eq:Tlip}.

Then our fully discrete Algorithm 1 in Section \ref{section:method} can be reformulated  as:
\begin{equation}\label{eq:fulldiscretelip}
 v^n_h=S_h(\size{t})T_h(\size{t})v_h^{n+1},\quad n=0,\cdots,N-1;\q
 v^N_h=P_h\phi.
\end{equation}

We further require the mild solution satisfies a stronger regularity in order to study the convergence of the scheme \eqref{eq:fulldiscretelip}.

\ms
\textbf{Assumption 3}.

There exists a unique mild solution $v\in C([0,T];L^2(\R^d))\cap B([0,T];H^1(\R^d))$ to \eqref{eq:semilinear}.

%

Our main result of this section is the following convergence theorem for  Algorithm 1.

\begin{Theorem}\label{thm:fullconvlip}
Under Assumptions 1 and 3, the solution of the scheme \eqref{eq:fulldiscretelip} converges to the mild solution $v$ of \eqref{eq:semilinear} in the sense that
$$\max_{0\le n\le N}\norm{v(t_n)-v_h^n}{}\to 0\q \textnormal{\it{as} $\size{t}\to 0$ \it{and} $h=O(\size{t}^\a)$ \it{for some} $\a>0$}.$$

\end{Theorem}

\begin{proof}

From the fact that for each $n=0,\cdots, N$ and $\size{t}$, $h>0$:
\begin{align*}
\norm{v(t_n)-v_h^n}{}&\le \norm{v(t_n)-P_hv(t_n)}{}+\norm{P_hv(t_n)-v_h^n}{}\lesssim  h \sup_{t\in [0,T]}\norm{v(t)}{H^1(\R^d)}+\max_{0\le n\le N}\norm{P_hv(t_n)-v_h^n}{},
\end{align*}
 we can easily see that it suffices to establish that as  $\size{t}\to 0$ and $h=O(\size{t}^\a)$ for some  $\a>0$
\begin{equation}\label{eq:lipfulltarget}
\max_{0\le n\le N}\norm{P_hv(t_n)-v_h^n}{}\to 0.
\end{equation}

Using the definition \eqref{eq:localerr} of the local truncation error  $R^{n+1}_{\size{t}}$ at $t_{n+1}$ and \eqref{eq:fulldiscretelip}, we can obtain a point-wise estimate for any given $x$ in $\R^d$   that
\begin{eqnarray}
P_hv(t_n,x)-v_h^n(x)&=&P_h\big(S(\size{t})T(\size{t})v(t_{n+1},x)+R^{n+1}_{\size{t}}(x)\big)-v_h^n(x)\nb\\
&=&P_hS(\size{t})T(\size{t})v(t_{n+1},x)-S_h(\size{t})T_h(\size{t})v_h^{n+1}(x)+P_hR^{n+1}_{\size{t}}(x)\nb\\
&=&\big(P_hS(\size{t})P_hT(\size{t})P_hv(t_{n+1},x)-S_h(\size{t})T_h(\size{t})v_h^{n+1}(x)\big)+P_hR^{n+1}_{\size{t}}(x)\nb\\
&&+P_hS(\size{t})T(\size{t})v(t_{n+1},x)-P_hS(\size{t})P_hT(\size{t})P_hv(t_{n+1},x).
\label{eq:projlip}
\end{eqnarray}
%
%
Then we readily see from the boundedness of $S(h)$, $P_h$ and the Lipschitz continuity of $T(h)$ that
\begin{align*}
\norm{P_hv(t_n)-v_h^n}{}
\lesssim &(1+L\size{t})\norm{P_hv(t_{n+1})-v_h^{n+1}}{}+\norm{R^{n+1}_{\size{t}}}{}\\
&+\norm{S(\size{t})T(\size{t})v(t_{n+1})-S(\size{t})P_hT(\size{t})P_hv(t_{n+1})}{}.
\end{align*}

Now let us choose $p$ to be the smallest integer such that $p\ge \max\{\frac{1}{\a},2\}$ and $q$ to be the H\"older conjugate of $p$. Then it follows from H\"older's inequality for sums that
\begin{align*}
\norm{P_hv(t_n)-v_h^n}{}^p
\lesssim &(1+L\size{t})\norm{P_hv(t_{n+1})-v_h^{n+1}}{}^p+\norm{R^{n+1}_{\size{t}}}{}^p\\
&+\norm{S(\size{t})T(\size{t})v(t_{n+1})-S(\size{t})P_hT(\size{t})P_hv(t_{n+1})}{}^p,
\end{align*}
which along with the discrete Gronwall inequality and $v^N_h=P_h\phi$ gives
\begin{align*}
\max_{0\le n\le N}\norm{P_hv(t_n)-v_h^n)}{}^p
&\lesssim \norm{P_h\phi-v_h^{N}}{}^p+\sum_{n=0}^{N-1}\norm{R^{n+1}_{\size{t}}}{}^p\\
&+\sum_{n=0}^{N-1}\norm{S(\size{t})T(\size{t})v(t_{n+1})-S(\size{t})P_hT(\size{t})P_hv(t_{n+1})}{}^p\\
&= \sum_{n=0}^{N-1}\norm{R^{n+1}_{\size{t}}}{}^p+\sum_{n=0}^{N-1}\norm{S(\size{t})T(\size{t})v(t_{n+1})-S(\size{t})P_hT(\size{t})P_hv(t_{n+1})}{}^p.
\end{align*}

We then proceed to estimate the last two terms in the above expression. We first observe that the consistency  \eqref{eq:timeconsistency} of the time discretization in Theorem \ref{thm:timelip} implies
$$\sum_{n=0}^{N-1}\norm{R^{n+1}_{\size{t}}}{}^p\le (\max_{0\le n\le N-1}\norm{R^{n+1}_{\size{t}}}{})^{p-1}\big(\sum_{n=0}^{N-1}\norm{R^{n+1}_{\size{t}}}{}\big)\le \big(\sum_{n=0}^{N-1}\norm{R^{n+1}_{\size{t}}}{}\big)^p\to 0 \q \textnormal{\it{as} $\size{t}\to 0$}.$$

On the other hand, it follows from the linearity of $P_h$ and $S(\size{t})$  that
\begin{align}\label{eq:lipfulltriangle}
\begin{split}
&S(\size{t})T(\size{t})v(t_{n+1})-S(\size{t})P_hT(\size{t})P_hv(t_{n+1})\\
=&S(\size{t})\big(T(\size{t})v(t_{n+1})-v(t_{n+1})\big)
-S(\size{t})P_h\big(T(\size{t})P_hv(t_{n+1})-P_hv(t_{n+1})\big)\\
&+S(\size{t})\big(v(t_{n+1})-P_hv(t_{n+1})\big).
\end{split}
\end{align}

Note the definition of $T(\size{t})$ at $t=t_{n+1}$ implies that $T(\size{t})v(x)-v(x)=-\size{t}g(t_{n+1},x,v(x))$ for any $v\in L^2(\R^d)$ and $x\in \R^d$, from which we derive for any $x\in \R^d$ that
\begin{align*}
&|S(\size{t})\big(T(\size{t})v(t_{n+1})-v(t_{n+1})\big)-S(\size{t})\big(T(\size{t})P_hv(t_{n+1})-P_hv(t_{n+1})\big)|\\
\lesssim & \size{t}S(\size{t})|g(t_{n+1},x,v(t_{n+1},x))-g(t_{n+1},x,P_hv(t_{n+1},x))|,
\end{align*}
which, along with the boundedness of $S(\size{t})$ and Lipschitz continuity of $g$ implies that
\begin{align*}
&\sum_{n=0}^{N-1}\norm{S(\size{t})\big(T(\size{t})v(t_{n+1})-v(t_{n+1})\big)-S(\size{t})\big(T(\size{t})P_hv(t_{n+1})-P_hv(t_{n+1})\big)}{}^p\\
\lesssim &\sum_{n=0}^{N-1}(\size{t})^p\norm{v(t_{n+1})-P_hv(t_{n+1})}{}^p.
\end{align*}
We then readily derive from \eqref{eq:lipfulltriangle} and H\"older's inequality that as
$\size{t}\to 0$ and $h=O(\size{t}^\a)$
\begin{align*}
&\sum_{n=0}^{N-1}\norm{S(\size{t})T(\size{t})v(t_{n+1})-S(\size{t})P_hT(\size{t})P_hv(t_{n+1})}{}^p\\
\lesssim &\sum_{n=0}^{N-1}\norm{S(\size{t})\big(T(\size{t})v(t_{n+1})-v(t_{n+1})\big)-S(\size{t})\big(T(\size{t})P_hv(t_{n+1})-P_hv(t_{n+1})\big)}{}^p\\
+&\sum_{n=0}^{N-1}\norm{S(\size{t})(I-P_h)v(t_{n+1})}{}^p\\
\lesssim& (h^p(\size{t})^{p-1}+h^{p-\frac{1}{\a}})T \big(\sup_{t\in [0,T]}\norm{v(t,\cdot)}{H^1(\R^d)}\big)^p\to 0\,.
\end{align*}

The above analyses enable us to conclude that as $\size{t}\to 0$ and $h=O(\size{t}^\a)$
$$\max_{0\le n\le N}\norm{P_hv(t_n)-v_h^n)}{}^p\to 0,$$
which implies \eqref{eq:lipfulltarget} and completes our proof.
\end{proof}

\subsection{Convergence for a maximal monotone driver $g$}\label{sec:alog2mono}
In this section, we demonstrate the numerical solution of Algorithm 2 (Section \ref{section:method}) provides
a good approximation to the mild solution $v(t)$ under the assumption that the driver is independent of the component $Z$ and satisfies a maximal monotonicity.
We shall first rigorously analyze the convergence of the semi-discrete scheme and then the fully discrete scheme.

\subsubsection{Convergence of the semi-discrete scheme}
We shall consider the driver $g$ to be  locally Lipschitz continuous and maximal monotone.

\ms
\textbf{Assumption 4}.
    \begin{enumerate}[{(a)}]
    \item For any given constant $c\ge 0$, there exists a constant $L(c)>0$ such that it holds for each two pairs $(t_1,u),(t_2,v)\in[0,T)\times L^2(\R^d)$ with $\norm{u}{}\le c,\norm{v}{}\le c$ that
        $$\norm{g(t_1,\cdot,u(\cdot))-g(t_2,\cdot,v(\cdot))}{}\le L(c)\big(|t_1-t_2|+\norm{u-v}{}\big).$$
    \item  There exists $\gamma\in \R$ such that the driver $g$ satisfies the monotonicity:
    \begin{equation}\label{eq:mono}
    (g(t,x,v)-g(t,x,w))(v-w)\ge \gamma|v-w|^2 ~~\quad \fa (t,x,v),(t,x,w)\in[0,T)\times \R^d\times \R^d\,.
    \end{equation}

    \item The driver $g$ satisfies the range condition:
    \begin{equation}\label{eq:range}
    \range(I+\lambda g(t,x,\cdot))=\R ~~\q \fa (t,x)\in [0,T)\times \R^d, \lambda>0.
    \end{equation}
     \end{enumerate}

We remark that the range condition \eqref{eq:range} ensures the solvability of the implicit scheme \eqref{eq:alphatkimp}. For our subsequent analysis, we introduce a semi-discrete version of Algorithm 2 (Section 3):
\begin{equation}\label{eq:timediscretemono}
 v^n=T(\size{t})S(\size{t})v^{n+1}, \q n=0,\cdots,N-1;\q  v^N=\phi,
\end{equation}
where operator $T(\size{t})$ is defined implicitly for any given $t\in [0,T)$ and $v\in L^2(\R^d)$:
\begin{equation}\label{eq:Tmono}
[T(\size{t})v](\cdot)+\size{t} g\big(t,\cdot,[T(\size{t})v](\cdot)\big)=v(\cdot).
\end{equation}
For each $n=0,\cdots,N-1$, we define the  local truncation error of the scheme \eqref{eq:timediscretemono} at $t_{n+1}$ as:
\begin{equation}\label{eq:localerrmono}
R^{n+1}_{\size{t}}(x)=v(t_n,x)-T(\size{t})S(\size{t})v(t_{n+1},x) \q \mbox{for} ~\textnormal{\it{a.e.} $x\in \R^d$}.
\end{equation}

The following Lipschitz property of operator $T(\size{t})$ is essential for the stability of \eqref{eq:timediscretemono}.

\begin{Proposition}\label{thm:T2}
Under Assumptions 1 and 4,  operator $T(h)$ is a (single-valued) Lipschitz continuous operator from $L^2(\R^d)$ into $L^2(\R^d)$ for small enough $\size{t}$ in the sense that for $|\gamma| \size{t}\le \frac{1}{2}$,
$$\norm{T(\size{t})v-T(\size{t})w}{}\le (1+2|\gamma| \size{t})\norm{v-w}{}\quad \fa v,w\in L^2(\R^d).$$
\end{Proposition}

\begin{proof}
For any given $\size{t}\in (0,T]$ and $v,w\in L^2(\R^d)$, the range condition implies $T(\size{t})v,T(\size{t})w\not=\emptyset$. Let $\hat{v}\in T(\size{t})v,\hat{w}\in T(\size{t})w$, we can obtain from  \eqref{eq:mono} that  for almost every $x$ in $\R^d$,
\begin{align*}
|\hat{v}(x)-\hat{w}(x)|^2
=&\big(v(x)-\size{t}g(t,x,\hat{v}(x))-w(x)+\size{t}g(t,x,\hat{w}(x)\big)(\hat{v}(x)-\hat{w}(x))\\
=&(v(x)-w(x))(\hat{v}(x)-\hat{w}(x))-\size{t}(g(t,x,\hat{v}(x))-g(t,x,\hat{w}(x)))(\hat{v}(x)-\hat{w}(x))\\
\le& |v(x)-w(x)||\hat{v}(x)-\hat{w}(x)|-\size{t}\gamma|\hat{v}(x)-\hat{w}(x)|^2,
\end{align*}
which along with the fact that $\frac{1}{1+\gamma \size{t}}\le \frac{1}{1-|\gamma| \size{t}}\le 1+2|\gamma| \size{t}$ for $|\gamma| \size{t}\le \frac{1}{2}$ implies
\begin{equation}\label{eq:monostab}
\norm{\hat{v}-\hat{w}}{}\le (1+2|\gamma| \size{t})\norm{v-w}{}.
\end{equation}

To establish that for any given $t\in [0,T)$, $T(\size{t})$ maps $L^2(\R^d)$ into $L^2(\R^d)$, we first introduce the functions $w_0(x)=g(t,x,0)$ and $h(x)=g(0,x,0)$ for $x\in \R^d$. It follows from the Lipschitz continuity of $g$ and the integrability of $h$ that
$$\norm{w_0}{}\le \norm{h}{}+\norm{w_0-h}{}\le \norm{h}{}+L(0)t,$$
which implies $w_0\in L^2(\R^d)$. Therefore, for any $\hat{v}\in T(\size{t})v$ with $v\in L^2(\R^d)$, by noticing that the zero function $0\in T(\size{t})w_0$, we see from \eqref{eq:monostab}  that for $|\gamma| \size{t}\le \frac{1}{2}$,
$$\norm{\hat{v}-0}{}\le (1+2|\gamma| \size{t})\norm{v-w_0}{}.$$
Now we can easily see from taking $v=w$ in \eqref{eq:monostab} that
operator $T(\size{t})$ is single-valued.
\end{proof}

An analogue of Theorem \ref{thm:timelip} concludes the consistency and convergence of the scheme  \eqref{eq:timediscretemono}.

\begin{Theorem} Under Assumptions 1 and 4, the scheme \eqref{eq:timediscretemono} is consistent in the sense that
\begin{equation}\label{eq:timeconsistencymono}
\sum_{n=0}^{N-1}\norm{R^{n+1}_{\size{t}}}{}\to 0 \q  \textnormal{\it{as} $N\to \infty$}.
\end{equation}
Moreover, the solution $v^n$ of the scheme converges to the mild solution $v$ in the sense that,
$$\max_{0\le n\le N}\norm{v(t_n)-v^n}{}\to 0\q  \textnormal{\it{as} $N\to \infty$}.$$
\end{Theorem}

\begin{proof}

Rearranging terms in \eqref{eq:localerrmono} leads to
\begin{equation}\label{eq:trunerr2}
v(t_n,x)=T(\size{t})S(\size{t})v(t_{n+1},x)+R^{n+1}_{\size{t}}(x)~~\q  \textnormal{for \it{a.e.} $x\in \R^d$}.
\end{equation}
Then for almost every $x$ in $\R^d$, subtracting \eqref{eq:trunerr2} from the scheme \eqref{eq:timediscretemono}, we obtain that
\begin{align*}
|v(t_n,x)-v^n(x)|&=|T(\size{t})S(\size{t})v(t_{n+1},x)+R_{n+1}(x)-T(\size{t})S(\size{t})v^{n+1}(x)|\\
&\le |T(\size{t})S(\size{t})v(t_{n+1},x)-T(\size{t})S(\size{t})v^{n+1}(x)|+|R_{n+1}(x)|.
\end{align*}
Taking $L^2$-norm on both sides of the above expression and using the boundedness of $S(\size{t})$ and Lemma \ref{thm:T2}, we have the following estimate of the numerical errors: for $|\gamma| \size{t}\le \frac{1}{2}$,
\begin{equation*}
\norm{v(t_n)-v^n}{}\lesssim 
(1+2|\gamma| \size{t})\norm{v(t_{n+1})-v^{n+1}}{}+\norm{R^{n+1}_{\size{t}}}{},
\end{equation*}
from which, the discrete Gronwall inequality and the condition $v^n=v(T)=\phi$ it follows that there exists $t_1>0$ such that for any $\size{t}<t_1$,
\begin{align*}
\max_{0\le n\le N}\norm{v(t_n)-v^n}{}\lesssim \norm{v(T)-v^N}{}+\sum_{n=0}^{N-1} \norm{R^{n+1}_{\size{t}}}{}=\sum_{n=0}^{N-1} \norm{R^{n+1}_{\size{t}}}{}.
\end{align*}

Now we can readily see that it suffices to establish the consistency \eqref{eq:timeconsistencymono} in order to conclude
the convergence of \eqref{eq:timediscretemono}.
Since the mild solution satisfies \eqref{eq:inteq}, we obtain for any given $x\in \R^d$:
\begin{align*}
v(t_n,x)&=S(\size{t})v(t_{n+1},x)-\int_{t_n}^{t_{n+1}} S(\tau-t_n)g(\tau,x,v(\tau,x))\,d\tau\\
&-\size{t}g(t_n,x,v(t_n,x))+\int_{t_n}^{t_{n+1}}g(t_n,x,v(t_n,x))\,d\tau,
\end{align*}
which along with the definition of $T(\size{t})$ gives
\begin{align*}
v(t_n,x)=T(\size{t})\Big\{S(\size{t})v(t_{n+1},x)-\int_{t_n}^{t_{n+1}} S(\tau-t_n)g(\tau,x,v(\tau,x))\,d\tau+\int_{t_n}^{t_{n+1}}g(t_n,x,v(t_n,x))\,d\tau\Big\}.
\end{align*}
Subtracting \eqref{eq:trunerr2} from the above equation, taking $L^2$-norm on both sides, using the inequality for Bochner integral and Lemma \ref{thm:T2} we obtain for $|\gamma| \size{t}\le \frac{1}{2}$ that
\begin{align*}
\norm{R^{n+1}_{\size{t}}}{}&\le
(1+2|\gamma| \size{t})\int_{t_n}^{t_{n+1}} \norm{S(\tau-t_n)g(\tau,\cdot,v(\tau,\cdot))-g(t_n,\cdot,v(t_n,\cdot))}{}\,d\tau\\
&\le (1+2|\gamma| \size{t})\int_{t_n}^{t_{n+1}} \norm{\big(S(\tau-t_n)-I\big)g(\tau,\cdot,v(\tau,\cdot))+\big(g(\tau,\cdot,v(\tau,\cdot))-g(t_n,\cdot,v(t_n,\cdot))\big)}{}\,d\tau\\
&\le (1+2|\gamma| \size{t})\int_{t_n}^{t_{n+1}} \norm{\big(S(\tau-t_n)-I\big)g(\tau,\cdot,v(\tau,\cdot))}{}+\norm{g(\tau,\cdot,v(\tau,\cdot))-g(t_n,\cdot,v(t_n,\cdot))}{}\,d\tau.
\end{align*}

Now we estimate the last two terms in the above expression. Since the mild solution $v\in C([0,T];L^2(\R^d))$, we have $\sup_{t\in[0,T]}\norm{v(t)}{}\le c$ for some constant $c\ge 0$. Hence the locally Lipschitz continuity of $g$ implies that there exists $L(c)>0$ such that
$$\norm{g(s,\cdot,v(s,\cdot))-g(t,\cdot,v(t,\cdot))}{}\le L(c)\big(|s-t|+\norm{v(s)-v(t)}{}\big)~~\quad \fa s, t\in [0,T).$$
Furthermore for any given $\epsilon>0$, there exists $0<t_0<\min(\epsilon, t_1)$ such that for any $\size{t}<t_0$ and $\tau\in [t_n,t_{n+1}]$, we have
$
\norm{v(\tau)-v(t_{n})}{}< \epsilon
$. 
These estimates directly imply for any $\size{t}<t_0$,
\begin{align*}
\norm{R^{n+1}_{\size{t}}}{}&\le \int_{t_n}^{t_{n+1}} \norm{\big(S(\tau-t_n)-I\big)g(\tau,\cdot,v(\tau,\cdot))}{}+\norm{g(\tau,\cdot,v(\tau,\cdot))-g(t_n,\cdot,v(t_n,\cdot))}{}\,d\tau\\
&\lesssim \int_{t_n}^{t_{n+1}} \norm{\big(S(\tau-t_n)-I\big)g(\tau,\cdot,v(\tau,\cdot))}{}\, d\tau+2\epsilon \size{t},
\end{align*}
from which it is clear that it suffices to establish
 \begin{equation}\label{eq:fmono}
 \sum_{n=0}^{N-1} \int_{t_n}^{t_{n+1}} \norm{\big(S(\tau-t_n)-I\big)|g(\tau,\cdot,v(\tau,\cdot))|}{}\,d\tau=
 \int_0^T f(\tau)\, d\tau\to 0 \q \textnormal{ as $N\to \infty$},
 \end{equation}
 where $f(\tau)\coloneqq \sum_{n=0}^{N-1} \norm{\big(S(\tau-t_n)-I\big)|g(\tau,\cdot,v(\tau,\cdot))|}{}1_{[t_n,t_{n+1})}(\tau)$ for each $\tau\in [0,T]$.
In fact, using the facts that $\sup_{t\in[0,T]}\norm{v(t)}{}\le c$ for some constant $c>0$, $S(t)$ are bounded and $g$ satisfies locally Lipschitz continuity, we obtain for each $\tau\in [0,T)$ that
 \begin{align*}
 f(\tau)\le L(c)( \tau+\norm{v(\tau)}{})+\norm{g(0,\cdot,0)}{},
 \end{align*}
then \eqref{eq:fmono} follows from an argument using the strong continuity of $S(t)$ and the Dominated Convergence Theorem,
which is similar to that in the proof of \eqref{eq:flip}.
 \end{proof}

\subsubsection{Convergence of the fully discrete scheme}

In this section, we establish the convergence of the fully discrete Algorithm 2 to the mild solution $v$ satisfying the variation of constants formula \eqref{eq:inteq}. We shall perform our analysis under the following extra assumption on the transition semigroup:

\ms
\textbf{Assumption 5}.
The family $\{S(t)\}_{0\le t\le T}$ maps $H^1(\R^d)$ into $H^1(\R^d)$.
And there exist a function $\rho(t):(0,\infty)\to (0,\infty)$ and  a real constant $\b>0$ with $\rho(t)=O(t^{-\b})$ as $t\to 0$ such that
    $$\norm{S(t)f}{H^1(\R^d)}\le \rho(t)\norm{f}{H^1(\R^d)}~~\quad \fa t>0, f\in H^1(\R^d).$$

We remark that Assumption 5 is satisfied by a large class of Markov processes, including It\^{o} diffusions associated with a uniformly elliptic operators $\gene$ in divergence form \cite{sheu1991some}, with $\b={1}/{2}$.

For our subsequent analysis, we shall 
reformulate Algorithm 2 in Section \ref{section:method} as:
\begin{equation}\label{eq:fulldiscretemono}
 v^n_h=T_h(\size{t})S_h(\size{t})v_h^{n+1},\quad n=0,\cdots,N-1; \q
 v^N_h=P_h\phi,
\end{equation}
where $T_h(\size{t})$,$S_h(\size{t})$ are defined in \eqref{eq:discreteop} with $T(\size{t})$ replaced by the implicit Euler scheme  \eqref{eq:Tmono}.

With \eqref{eq:fulldiscretemono} in hand, we are now ready to conclude the convergence of Algorithm 2.

\begin{Theorem}\label{thm:fullconvmono}
Under Assumptions 1, 3, 4 and 5, the solution of the scheme \eqref{eq:fulldiscretemono} converges to the mild solution $v$ of \eqref{eq:semilinear}. More precisely, we have
$$\max_{0\le n\le N}\norm{v(t_n)-v_h^n}{}\to 0\q \textnormal{\it{as} $\size{t}\to 0$ \it{and} $h=O(\size{t}^\a)$ \it{for some} $\a>\b$},$$
where $\b$ is the constant given in Assumption 5.

\end{Theorem}

\begin{proof}
We can  infer using the same argument as that of Theorem \ref{thm:fullconvlip} that it suffices to establish
\begin{equation}\label{eq:monofulltarget}
\max_{0\le n\le N}\norm{P_hv(t_n)-v_h^n}{}\to 0\q \textnormal{\it{as} $\size{t}\to 0$ \it{and} $h=O(\size{t}^\a)$ \it{for some} $\a>\b$}.
\end{equation}

We see from the definition \eqref{eq:localerrmono}  
and \eqref{eq:fulldiscretemono} that for almost every $x\in \R^d$,
\begin{align*}
P_hv(t_n,x)-v_h^n(x)&=P_h\big(T(\size{t})S(\size{t})v(t_{n+1},x)+R^{n+1}_{\size{t}}(x)\big)-v_h^n(x)\\
&=P_hT(\size{t})S(\size{t})v(t_{n+1},x)+P_hR^{n+1}_{\size{t}}(x)-P_hT(\size{t})P_hS(\size{t})v_h^{n+1}(x)\\
&=P_h\big(T(\size{t})S(\size{t})v(t_{n+1},x)-T(\size{t})P_hS(\size{t})v_h^{n+1}(x) \big)+P_hR^{n+1}_{\size{t}}(x).
\end{align*}
Then taking the $L^2$-norm, using the boundedness of $S(\size{t})$ and the Lipschitz continuity of $T(\size{t})$,
we deduce that there exists $t_1>0$ such that for any $\size{t}<t_1$ and $n=0,\cdots,N-1$:
\begin{align*}
\norm{P_hv(t_n)-v_h^n}{}
\lesssim &(1+|\gamma|\size{t})\norm{P_hS(\size{t})v_h^{n+1}-S(\size{t})v(t_{n+1})}{}+\norm{R^{n+1}_{\size{t}}}{}\\
\le &(1+|\gamma|\size{t})\norm{P_hS(\size{t})v_h^{n+1}-P_hS(\size{t})P_hv(t_{n+1})}{}\\
&+\norm{P_hS(\size{t})P_hv(t_{n+1})-S(\size{t})v(t_{n+1})}{}+\norm{R^{n+1}_{\size{t}}}{}\\
\lesssim &(1+|\gamma|\size{t})\norm{P_hv(t_{n+1})-v_h^{n+1}}{}+\norm{P_hS(\size{t})P_hv(t_{n+1})-S(\size{t})v(t_{n+1})}{}+\norm{R^{n+1}_{\size{t}}}{}.
\end{align*}

Now let us choose $p$ to be the smallest integer such that $p> \max\{(\a-\b)^{-1},2\}$ and $q$ to be the H\"older conjugate of $p$. Note this choice of $p$ ensures $p-{\a}^{-1}>0$ and $(1-{\b}{\a}^{-1})p-{\a}^{-1}>0$. Then H\"older's inequality for sums asserts that
\begin{align*}
\norm{P_hv(t_n)-v_h^n}{}^p\lesssim (1+\size{t})\norm{P_hv(t_{n+1})-v_h^{n+1}}{}^p+\norm{P_hS(\size{t})P_hv(t_{n+1})-S(\size{t})v(t_{n+1})}{}^p+\norm{R^{n+1}_{\size{t}}}{}^p.
\end{align*}
By further applying the discrete Grownwall inequality to the above estimate, we obtain that
\begin{align*}
\max_{0\le n\le N}\norm{P_hv(t_n)-v_h^n)}{}^p
&\lesssim  \norm{P_h\phi-v_h^{N}}{}^p+\sum_{n=0}^{N-1}\norm{P_hS(\size{t})P_hv(t_{n+1})-S(\size{t})v(t_{n+1})}{}^p+\sum_{n=0}^{N-1}\norm{R^{n+1}_{\size{t}}}{}^p\\
\le &\sum_{n=0}^{N-1}\big(\norm{P_hS(\size{t})(I-P_h)v(t_{n+1})}{}^p+\norm{(I-P_h)S(\size{t})v(t_{n+1})}{}^p\big)
+\sum_{n=0}^{N-1}\norm{R^{n+1}_{\size{t}}}{}^p.
\end{align*}

We then bound the last three terms in the above expression. First one has that as $h=O(\size{t}^{\a})$
\begin{align*}
&\norm{P_hS(\size{t})(I-P_h)v(t_{n+1})}{}^p\lesssim h^{p-\frac{1}{\a}}\size{t}\norm{v(t_{n+1})}{H^1(\R^d)}^p,\\
&\norm{(I-P_h)S(\size{t})v(t_{n+1})}{}^p\lesssim h^{p-\frac{1}{\a}}\size{t}\norm{S(\size{t})v(t_{n+1})}{H^1(\R^d)}^p.
\end{align*}
Moreover, we obtain from the consistency  \eqref{eq:timeconsistencymono} of time discretization that
\begin{align*}
\sum_{n=0}^{N-1}\norm{R^{n+1}_{\size{t}}}{}^p\le (\max_{0\le n\le N-1}\norm{R^{n+1}_{\size{t}}}{})^{p-1}\big(\sum_{n=0}^{N-1}\norm{R^{n+1}_{\size{t}}}{}\big)\le \big(\sum_{n=0}^{N-1}\norm{R^{n+1}_{\size{t}}}{}\big)^p\to 0 \q \textnormal{\it{as} $\size{t}\to 0$}.
\end{align*}
These analyses along with Assumption 5 lead to the estimate:
\begin{align*}
&\max_{0\le n\le N}\norm{P_hv(t_n)-v_h^n)}{}^p\\
\lesssim &
\sum_{n=0}^{N-1} h^{p-{\a}^{-1}}\size{t}\norm{v(t_{n+1})}{H^1(\R^d)}^p+\sum_{n=0}^{N-1}h^{p-{\a}^{-1}}\size{t}\big(\rho(\size{t})\big)^p\norm{v(t_{n+1})}{H^1(\R^d)}^p+ \big(\sum_{n=0}^{N-1}\norm{R^{n+1}_{\size{t}}}{}\big)^p\\
\lesssim& h^{p-{\a}^{-1}} T\big(\sup_{t\in [0,T]}\norm{v(t)}{H^1(\R^d)}\big)^p+h^{(1-{\b}{\a}^{-1})p-{\a}^{-1}}T\big(\sup_{t\in [0,T]}\norm{v(t)}{H^1(\R^d)}\big)^p+\big(\sum_{n=0}^{N-1}\norm{R^{n+1}_{\size{t}}}{}\big)^p \to 0
\end{align*}
as $\size{t}\to 0$ and $h=O(\size{t}^\a)$, which enables us to conclude the desired result \eqref{eq:monofulltarget}.
\end{proof}


\color{black}
%
%

\section{Numerical experiments}\l{sec:numerical}
In this section we shall present several numerical examples to demonstrate the effectiveness of of Algorithm 2 in Section \ref{section:method} for solving FBSDEs whose drivers are uniformly Lipschitz continuous on $Z$ but satisfy various regularity assumptions on $Y$.

For each example below, we shall take a computational domain such that the exit probability of
the forward process is as small as $10^{-3}$. 
The transition probability \eqref{eq:Pij} and the spatial derivatives in \eqref{eq:beta} are approximated by quadrature rules and central differences, respectively. 
When analytic solutions are known, we examine the convergence of the numerical solution $v_h$ to the exact solution $v$
using relative $L^2$-norm errors:
\begin{equation}\label{eq:relative}
E_h(v)=\frac{\max_{0\le k\le N}\|v_h^k-v(t_k)\|}{\max_{0\le k\le N}\|v(t_k)\|}.
\end{equation}
%
We will also consider the accuracy of our numerical solution $(\hat{Y}_0,\hat{Z}_0)$ at the initial time $t=0$, which is a commonly used criterion to check the performance of a numerical scheme for FBSDEs.
We remark that the drivers  in Examples 1 and 2 have a Lipschitz dependence on the component $Y$, hence a fixed point iterative method could be applied to solve the nonlinear equation \eqref{eq:alphatkimp} in Algorithm 2, whereas in Example 3 the driver $g$ is only locally Lipscthiz continuous with respect to $Y$, therefore Newton's method shall be utilized to solve the corresponding nonlinear equations.

\textbf{Example 1}.
Consider the following FBSDE:
$$\begin{aligncases}
dX_t&=dB_t,&& X_0=0,\\
dY_t&=\frac{\bigg(-\big(\frac{\pi}{2}\big)^2(Y_t-1)+2Z_t\bigg)}{(Y_t-1)^2+\big(\frac{Z_t}{\pi/2}\big)^2+1} dt+Z_t dB_t,&& Y_T=\sin(\frac{\pi}{2}(B_T+T))+1.
\end{aligncases}$$

The corresponding semilinear parabolic PDE \eqref{eq:semilinear} is given by $v(T,x)=\sin(\frac{\pi}{2}(x+t))+1$ and
\bb\label{eq:ex1}
v_t+\frac{1}{2}v_{xx}=\frac{(-(\pi/2)^2(v-1)+2v_x)}{(v-1)^2+(2v_x /\pi)^2+1},\q (t,x)\in (0,T)\times \R.
\ee
The analytic solution to this system is given by
$v(t,x)=\sin(\frac{\pi}{2}(x+t))+1$ for $(t,x)\in [0,T]\times \R$.
We will take the terminal time $T=2$ and the computational domain
$\dom= (-4,4)$.
%
%
%

From the analytic solution, we know the exact value of $(Y_0,Z_0)=(1,\frac{\pi}{2})$.
%
The relative $L^2$-norm errors
of $v_h$, $\hat{Y}_0$ and $\hat{Z}_0$ respectively
with different mesh sizes are shown in Table \ref{table:ex1}, illustrating a first-order convergence of our scheme. Figure \ref{fig:valuefunctions} (left) presents the numerical solution to \eqref{eq:ex1} with  mesh size $h=0.01$, which clearly captures the periodic behaviour of the exact solution.

\begin{tablehere}
\centering
\begin{tabular}{||l|c|c|c|c|c|c||}\hline
$h$ &  0.08 & 0.04 & 0.02 &0.01 & 0.005 & 0.0025 \\ \hline
 $v_h$&    0.1305  &  0.0654  &  0.0327 &   0.0163  &  0.0082  &  0.0041\\\hline
 $\hat{Y}_0$ &  0.1445  &  0.0714  &  0.0355  &  0.0177 &   0.0088 &   0.0044 \\\hline
$\hat{Z}_0$ &0.0649   & 0.0346  &  0.0178 &   0.0090 &   0.0046 &   0.0023 \\\hline
\end{tabular}

\caption{Relative errors of the value function $v$, $Y_0$ and $Z_0$ in Example 5.1.}
\label{table:ex1}
\end{tablehere}



\textbf{Example 2}.
This example considers a FBSDE that models an option under a market
with a borrowing interest rate $R$ and a (possibly different) return rate $r$ of the bond \cite{zhao2006new}
and three different sets of parameters and terminal conditions:
\begin{equation}\label{eq:ex2}
\begin{aligncases}
dX_t&=\mu X_td\tau+\sigma X_tdB_t,\q X_0=x_0,\\dY_t&=\big(rY_t+\frac{\mu-r}{\sigma} Z_t+(R-r)\min(Y_t-\frac{Z_t}{\sigma},0)\big) dt+Z_t dB_t.
\end{aligncases}
\end{equation}
%

\textbf{Example 2.1:} straddle with same interest rate, terminal condition  $Y_T=|X_T-K|$
and the following parameters \cite{bender2005forward}:

\begin{tablehere}
\centering
\begin{tabular}{||c|c|c|c|c|c|c||}\hline
$\mu$ &  $\sigma$ & $r$ & $R$ & $T$ & $x_0$ & $K$ \\ \hline
0.05 & 0.2 & 0.01& 0.01& 2& 1& 1\\\hline
\end{tabular}
\caption{Problem parameters  in Example 2.1.}
\end{tablehere}

The analytic solution to the BSDE \eqref{eq:ex2}
is given by the Black-Scholes formula:
\begin{equation}\label{eq:bsstradle}
\begin{aligncases}
Y_t&=v(t,X_t)=X_t\big(2N(d_1)-1\big)-Ke^{-r(T-t)}\big(2N(d_2)-1),\\
Z_t&=\sigma X_t \p_x v(v,X_t)=\sigma X_t\big(2N(d_1)-1\big),\\
d_1&=\dfrac{\log(X_t/K)+(r+\frac{1}{2}\sigma^2)(T-t)}{\sigma\sqrt{T-t}},
\quad d_2=d_1-\sigma\sqrt{T-t}.
\end{aligncases}
\end{equation}

We  take the computational domain $\dom=(0,3.5)$.
Table \ref{table:ex2straddlev} lists the relative errors of $v_h$ with different mesh sizes,
which show a first-order convergence of our scheme.

\begin{tablehere}
\centering
\begin{tabular}{||l|c|c|c|c|c|c||}\hline
$h$ &  0.08 & 0.04 & 0.02 &0.01 & 0.005  \\ \hline
$v_h$ &   0.0377  &  0.0190 &   0.0096   & 0.0048  &  0.0025\\\hline
 \end{tabular}
\caption{Relative errors of the value function $v$ in Example 2.1.}
\label{table:ex2straddlev}
\end{tablehere}

We further study the accuracy of the numerical solution $(\hat{Y}_0,\hat{Z}_0)$, where $\hat{Y}_0$ approximates the current fair price of the staddle. From the analytic solution \eqref{eq:bsstradle}, we obtain the exact value $(Y_0,Z_0)=(0.2233,0.0336)$. We then demonstrate the convergence of the relative errors in Table \ref{table:ex2straddley_0} and plot the numerical value function with mesh size $h=0.01$ in Figure \ref{fig:valuefunctions} (middle).

\begin{tablehere}
\centering
\begin{tabular}{||l|c|c|c|c|c|c||}\hline
$h$ &   0.04 & 0.02 &0.01 & 0.005 & 0.0025 &0.00125\\ \hline
$\hat{Y}_0$ &   0.0112  &  0.0064  &  0.0034  &  0.0019  &  0.0011  &  0.0006\\\hline
$\hat{Z}_0$ &0.3750   & 0.1875  &  0.0923 &   0.0476  &  0.0238 &   0.0119 \\\hline
\end{tabular}
\caption{Relative errors of $(\hat{Y}_0,\hat{Z}_0)$ in Example 2.1.}
\label{table:ex2straddley_0}
\end{tablehere}

%
%
%
%
\textbf{Example 2.2:} Call option with different interest rates, terminal condition  $Y_T=\max(X_T-K,0)$
and the following parameters :


\begin{tablehere}
\centering
\begin{tabular}{||c|c|c|c|c|c|c||}\hline
$\mu$ &  $\sigma$ & $r$ & $R$ & $T$ & $x_0$ & $K$ \\ \hline
0.06 & 0.2 & 0.04& 0.06& 2& 1& 1\\\hline
\end{tabular}
\caption{Problem parameters in Example 2.2.}
\label{table:paracall}
\end{tablehere}

By the replicating strategy of a call option, we know the analytic solution is given by the Black-Scholes formula evaluated with the interest rate $R$ \cite{gobet2005regression}:
\begin{equation}\label{eq:bscall}
\begin{aligncases}
Y_t&=v(t,X_t)=X_tN(d_1)-Ke^{-R(T-t)}N(d_2),\quad &&
Z_t=\sigma X_t \p_x v(t,X_t)=X_tN(d_1)\sigma,\\
d_1&=\frac{\log(X_t/K)+(R+\frac{1}{2}\sigma^2)(T-t)}{\sigma\sqrt{T-t}},\quad &&
d_2=d_1-\sigma\sqrt{T-t}.
\end{aligncases}
\end{equation}

This is a nonlinear problem with a known analytic solution. we shall take the computational domain $\dom=(0,4)$ to test our scheme.
Table \ref{table:ex2callv} lists the relative errors of $v_h$ with respect to different mesh sizes, and indicates that for a non-smooth driver $g$ and terminal condition $\phi$, the first-order convergence rate of $v_h$ may not be guaranteed.

\begin{tablehere}
\centering
\begin{tabular}{||l|c|c|c|c|c|c||}\hline
$h$ &  0.08 & 0.04 & 0.02 &0.01 & 0.005  \\ \hline
$v_h$ &   0.0335  &  0.0171 &   0.0093  &  0.0059 &   0.0047\\\hline
 \end{tabular}
\caption{Relative errors of the value function $v$ in Example 2.2.}
\label{table:ex2callv}
\end{tablehere}

We then investigate the accuracy of the numerical solution of the current fair price. We infer from \eqref{eq:bscall} directly that the exact value of $(Y_0,Z_0)$ is $(0.1720,0.1428)$.  Table \ref{table:ex2cally_0} lists their relative errors and indicates a first-order convergence of both terms. Figure \ref{fig:valuefunctions} (right) plots the numerical value function  with mesh size $h=0.01$.

\medskip
\begin{tablehere}
\centering
\begin{tabular}{||l|c|c|c|c|c|c||}\hline
$h$ &   0.04 & 0.02 &0.01 & 0.005 & 0.0025 &0.00125\\ \hline
$\hat{Y}_0$ &    0.0819  &  0.0413  &  0.0208 &   0.0105 &   0.0053 &   0.0027\\\hline
$\hat{Z}_0$ &0.0581 &   0.0280 &   0.0140  &  0.0070  &  0.0035 &   0.0015 \\\hline
\end{tabular}
\caption{Relative errors of $(\hat{Y}_0,\hat{Z}_0)$ in Example 2.2.}
\label{table:ex2cally_0}
\end{tablehere}

We finally examine the impact of the computational domain on  the numerical solutions. 
We perform our computations with mesh size $h=0.00125$ on the domain $(0,M)$ with different $M$, and list the relative errors of numerical solutions in Table \ref{table:ex2cally_0diffM}, which illustrates that for the numerical examples presented above, the computational domain  $(0,4)$ is sufficiently large and the errors caused by the domain truncation are almost neglectable. 

\medskip
\begin{tablehere}
\centering
\begin{tabular}{||l|c|c|c|c||}\hline
$M$ &  3 & 4 & 5 & 6 \\ \hline
$\hat{Y}_0$ &    0.003040  &  0.002732  &  0.002728 &   0.002728    \\\hline
$\hat{Z}_0$ &0.002561 &   0.001502 &   0.001484  &  0.001484    \\\hline
\end{tabular}
\caption{Relative errors of $(\hat{Y}_0,\hat{Z}_0)$ for different computational domains in Example 2.2.}
\label{table:ex2cally_0diffM}
\end{tablehere}

\color{black}


\begin{figurehere}
    \centering
    \includegraphics[width=5cm,height=4cm]{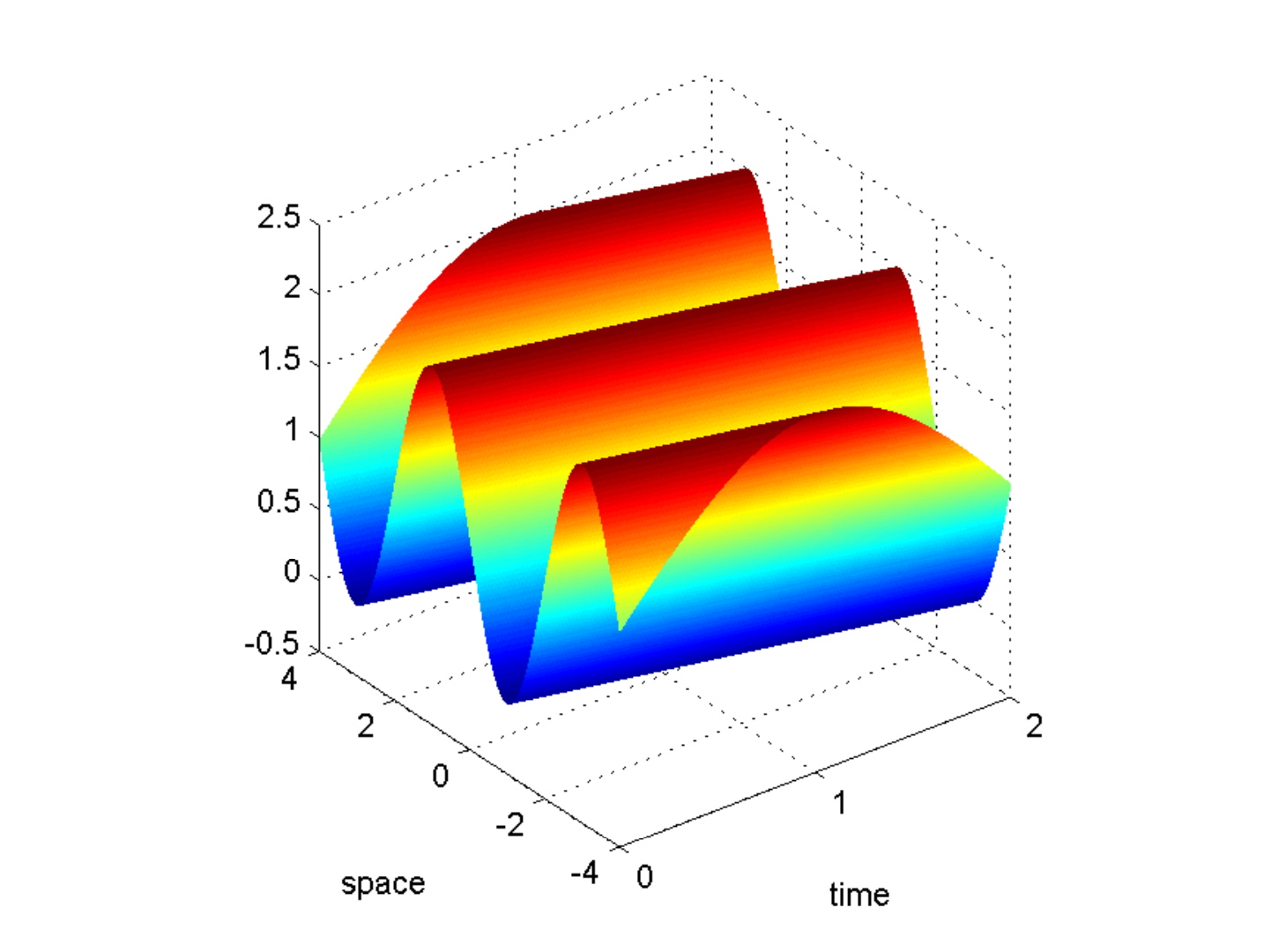}
    \includegraphics[width=5cm,height=4cm]{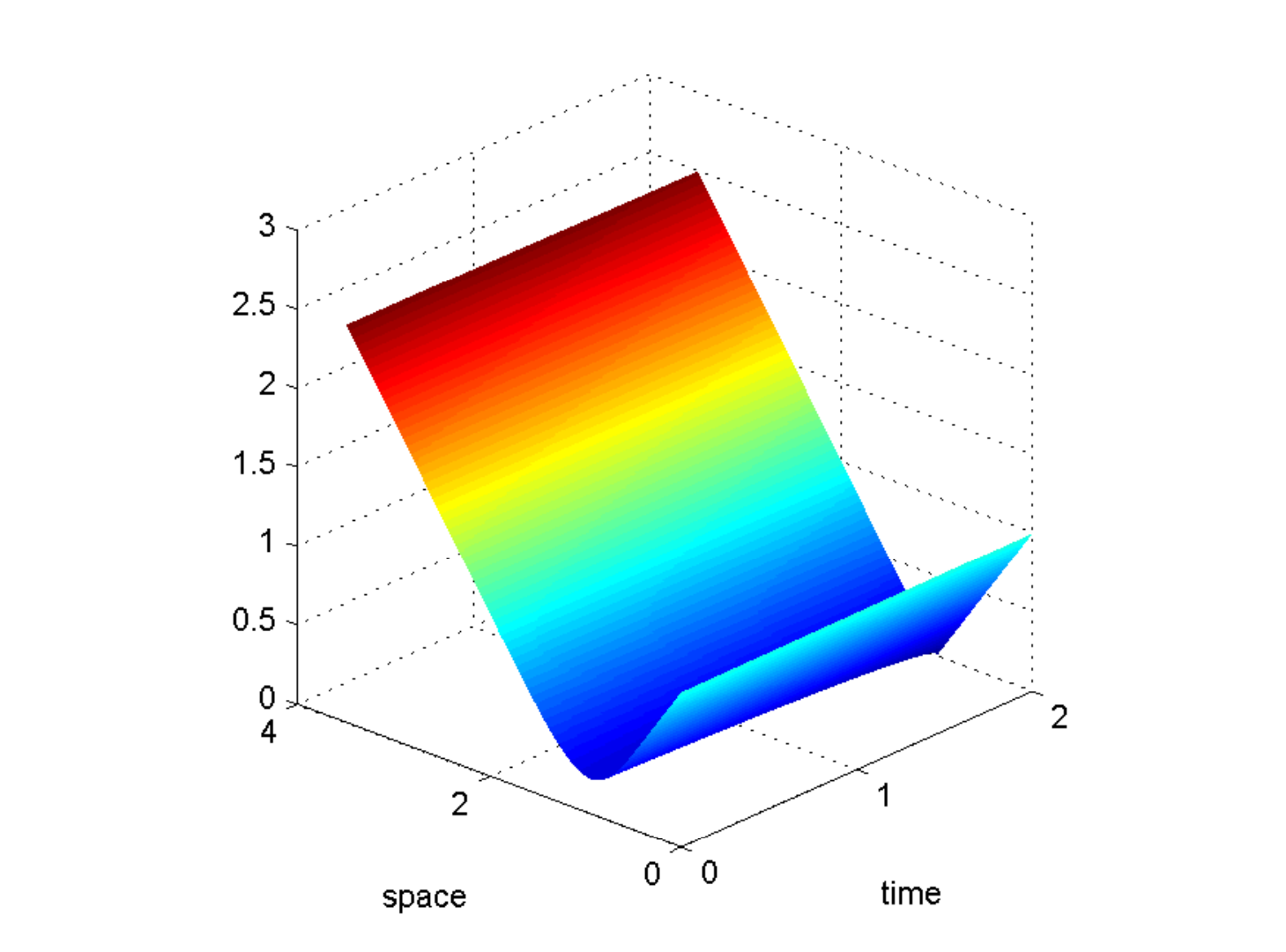}
    \includegraphics[width=5cm,height=4cm]{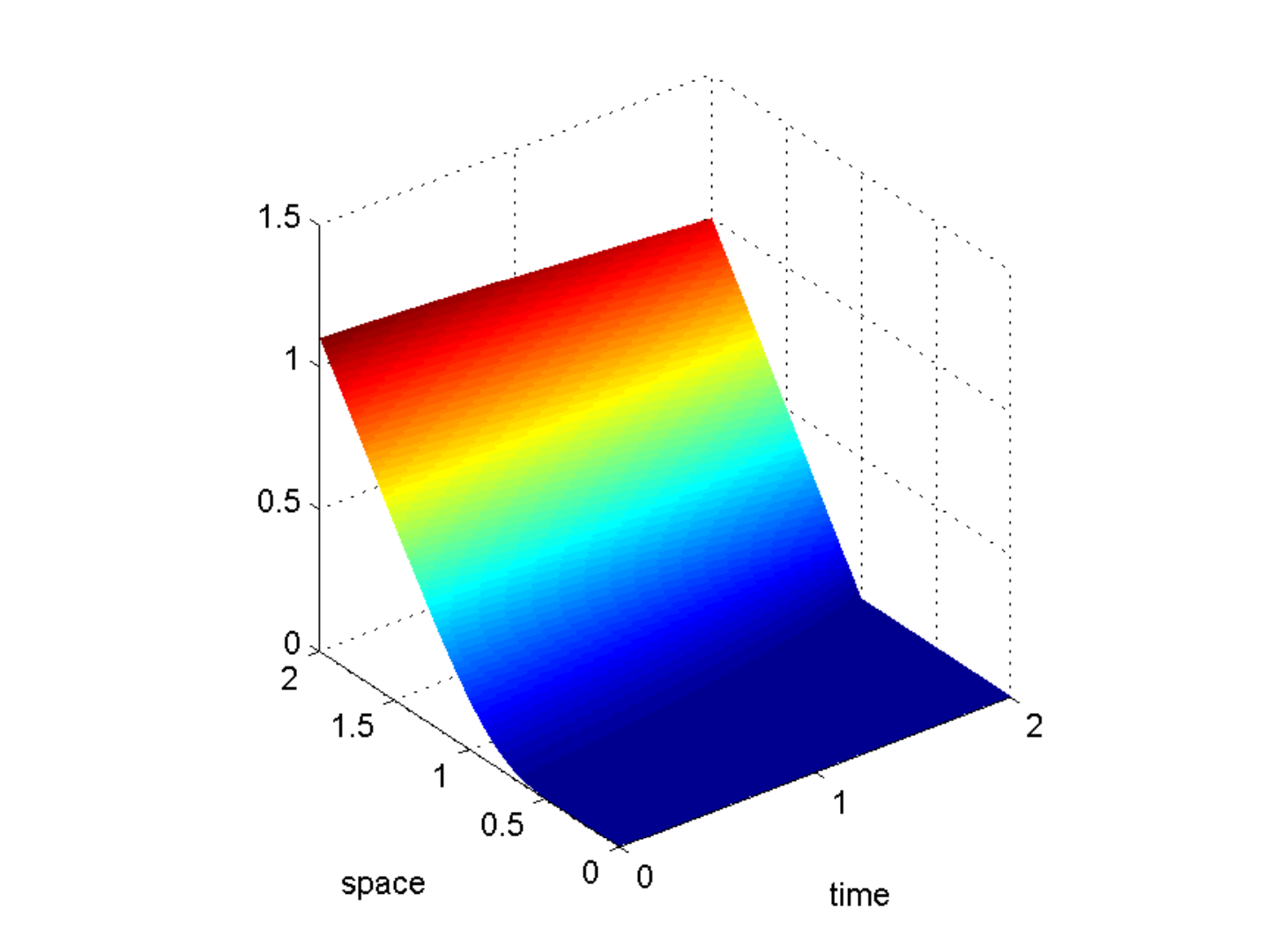}
    \caption{Numerical results of  the value functions in Examples 1, 2.1 and 2.2 from left to right.}
    \label{fig:valuefunctions}
 \end{figurehere}

\ms
\textbf{Example 2.3:} Call combination with different interest rates, terminal condition  $Y_T=\max(X_T-K_1,0)-2\max(X_T-K_2,0)$
and the following parameters :

\begin{tablehere}
\centering
\begin{tabular}{||c|c|c|c|c|c|c|c||}\hline
$\mu$ &  $\sigma$ & $r$ & $R$ & $T$ & $x_0$ & $K_1$ &$K_2$\\ \hline
0.06 & 0.2 & 0.04& 0.06& 2& 1& 0.95&1.05\\\hline
\end{tabular}
\caption{Problem parameters in Example 2.3.}
\label{table:paracombinedcall}
\end{tablehere}

The solution to this example is not provided by the Black-Scholes formula. The nonlinearity of the driver $g$ has a real impact on the value function $v$ and also $Y_0$. A reference price  suggested in \cite{gobet2005regression} is $Y_0=0.0295$. We shall price this option with the computational domain $(0,4)$.
Table \ref{table:ex2combinedcall} contains the numerical solutions $\hat{Y_0}$ with respect to different mesh sizes, and illustrates a good agreement between our results and the reference price.

\begin{tablehere}
\centering
\begin{tabular}{||l|c|c|c|c|c|c||}\hline
$h$ &   1/32&1/64&1/128&1/256&1/512 &1/1024\\ \hline
$\hat{Y}_0$    & 0.02849 & 0.02920 & 0.02943 & 0.02951 & 0.02955 &0.02956 \\\hline
 \end{tabular}
\caption{Numerical prices of the call combination.}
\label{table:ex2combinedcall}
\end{tablehere}


\textbf{Example 3}.
We end this section by considering a nonlinear FBSDE whose driver $g$ is locally Lipschitz continuous and maximal monotone on the component $Y$:
$$\begin{aligncases}
dX_t&=dB_t,\q X_0=0,\\
dY_t&=\bigg(Y_t^3-\frac{\pi^2}{8}(Y_t-1)+Z_t-\big(\sin(\frac{\pi}{2}(B_t+t))+1\big)^3 \bigg) dt+Z_t dB_t,\\
Y_T&=\sin(\frac{\pi}{2}(B_T+T))+1.
\end{aligncases}$$

The corresponding semilinear parabolic PDE \eqref{eq:semilinear} is given by $v(T,x)=\sin(\frac{\pi}{2}(x+t))+1$ and
$$
v_t+\frac{1}{2}v_{xx}=v^3-\frac{\pi^2}{8}(v-1)+v_x-\big(\sin(\frac{\pi}{2}(x+t))+1\big)^3,\q (t,x)\in (0,T)\times \R,$$
whose analytic solution is the same as that to \eqref{eq:ex1}.

We take the computational domain $\dom= (-4,4)$. Table \ref{table:ex3} contains the relative $L^2$-norm errors of $v_h$, $\hat{Y}_0$ and $\hat{Z}_0$ respectively with respect to different mesh sizes, which indicates a first-order convergence of our method.

\begin{tablehere}
\centering
\begin{tabular}{||l|c|c|c|c|c|c||}\hline
$h$ &   0.04 & 0.02 &0.01 & 0.005 & 0.0025 &0.00125 \\ \hline
$v_h$ & 0.09303 & 0.04323 & 0.02082 & 0.01022 & 0.005060 & 0.002518 \\\hline
$\hat{Y}_0$ &   0.1240  &  0.0571  &  0.0274  &  0.0135  &  0.0067 &   0.0033\\\hline
$\hat{Z}_0$ & 0.1429   & 0.0646   & 0.0308   & 0.0150   & 0.0075  &  0.0037\\\hline
\end{tabular}
\caption{Relative errors of the value function $v$, $Y_0$ and $Z_0$ in Example 5.3.}
\label{table:ex3}
\end{tablehere}

%

\medskip

\end{document}